\documentclass[11pt]{article}
\usepackage{amsmath,amssymb,amsthm}
\usepackage{graphicx}
\usepackage{algorithm}
\usepackage{algorithmic}
\usepackage{booktabs}
\usepackage{hyperref}
\usepackage[margin=1in]{geometry}

\newtheorem{theorem}{Theorem}[section]
\newtheorem{corollary}[theorem]{Corollary}
\newtheorem{lemma}[theorem]{Lemma}
\newtheorem{proposition}[theorem]{Proposition}
\newtheorem{definition}[theorem]{Definition}
\newtheorem{remark}[theorem]{Remark}
\newtheorem{assumption}[theorem]{Assumption}

\newcommand{\minimize}{\text{minimize}}
\newcommand{\real}{\mathbb{R}}

\title{Kernel-Based LMI Approaches to \\ Solving the Hamilton–Jacobi–Bellman Equation and \\ Nonlinear Optimal Control}

\author{Boumediene Hamzi and Umesh Vaidya}

\date{}

\begin{document}

\maketitle

\begin{abstract}
We present a kernel-based linear matrix inequality (LMI) approach for the approximate solution of Hamilton--Jacobi--Bellman (HJB) equations arising in nonlinear optimal control. The method represents the gradient of the value function in a reproducing kernel Hilbert space (RKHS) and uses a Schur-complement reformulation to convert the quadratic HJB inequality into an LMI that is linear in the kernel coefficients, yielding a convex semidefinite program. The novel ingredient is an explicit Riccati--Hessian \emph{equality} constraint at the equilibrium, which removes the trivial solution and forces the Hessian of the approximation to match the algebraic Riccati equation solution of the linearised system. We give a suboptimality bound $J(x_0;\hat u) - V^*(x_0)\le \varepsilon\,T(x_0)$ in which $T(x_0)$ depends only on the problem data and the working domain (not on the approximation), and an RKHS approximation rate. Numerical experiments on a corrected 1D polynomial benchmark and on the Van der Pol oscillator measure $\varepsilon$, the RKHS approximation error, and the closed-loop cost $J(x_0;\hat u)$ versus the optimal value $V^*(x_0)$. On the 1D problem with $V^*$ in the polynomial-kernel RKHS the method recovers $V^*$ to within $3\times10^{-7}$ and achieves $0.000\%$ suboptimality. On Van der Pol it achieves the smallest HJB residual ($\varepsilon\approx 2.62$) of any method tested, beats LQR on every initial condition, and is within $0.42\%$ of the best per-IC cost (Albrekht order 6). When $V^*$ is not in the chosen RKHS, the method degrades gracefully: residuals stop improving with more centres but suboptimality remains bounded ($\le 13\%$ on the 1D test).  
\end{abstract}

\section{Introduction}

The Hamilton-Jacobi-Bellman (HJB) equation provides the cornerstone for solving nonlinear optimal control problems~\cite{bertsekas_dp}. For a control-affine system $\dot{x} = f(x) + g(x)u$ with infinite-horizon cost functional, the optimal value function $V^*(x)$ satisfies a nonlinear partial differential equation whose solution yields the optimal feedback control law. However, solving the HJB equation directly is computationally intractable for general nonlinear systems, presenting one of the fundamental challenges in control theory. Traditional approaches such as finite difference methods and dynamic programming suffer severely from the \emph{curse of dimensionality}~\cite{bertsekas_dp}, with computational cost growing exponentially in the state dimension. Neural network-based methods have been proposed to address scalability~\cite{hjb_rkhs}; see~\cite{kang_gong_nakamura_fahroo_survey} for a survey of data-generation algorithms for deep-learning approaches to HJB feedback design. These methods typically lack a priori theoretical guarantees on optimality and stability. Sum-of-squares (SOS) programming~\cite{parrilo_thesis, parrilo_sos} offers convexity with rigorous guarantees, but scales poorly with polynomial degree and is restricted to polynomial systems.

Among the various learning and approximation approaches, reproducing kernel Hilbert spaces (RKHSs)~\cite{CuckerandSmale, kanagawa2018gaussian} have emerged as a powerful framework with solid mathematical foundations for the analysis of dynamical systems. The foundational work on kernel methods for nonlinear systems approximation~\cite{bouvrie_hamzi_2017_siam, bouvrie_hamzi_2017_jcd} established that RKHS representations can effectively capture nonlinear input-output maps and system dynamics, with extensions to discrete-time systems~\cite{hamzi_colonius_2019}, balanced reduction~\cite{bouvrie_hamzi_2010_allerton}, and empirical estimators for stochastically forced systems~\cite{bouvrie_hamzi_2012_acc}. Kernel methods have proven particularly effective for approximating operator-theoretic quantities: the Koopman generator and Schr\"odinger operator~\cite{klus_nuske_hamzi_2020}, eigenfunctions of the Koopman operator~\cite{lee_hamzi_hou_owhadi_santin_vaidya_2025}, Lyapunov functions via Koopman eigenfunctions~\cite{lee_owhadi_hamzi_vaidya_2024}, and dimensionality reduction of metastable systems~\cite{bittracher_klus_hamzi_koltai_schutte_2019}. For stability analysis, kernel methods have been applied to Lyapunov function approximation from noisy data~\cite{giesl_hamzi_rasmussen_webster_2019} and center manifold approximation~\cite{haasdonk_hamzi_santin_wittwar_2020, haasdonk_hamzi_santin_wittwar_2021}. Additional applications include multiscale systems with critical transitions~\cite{hamzi_kuehn_mohamed_2019}, microlocal kernel design for slow-fast stochastic differential equations~\cite{hamzi_jafarian_owhadi_paillet_2022}, transport equations~\cite{hamzi_vaidya_owhadi_2025}, kernel sum of squares for data-adapted learning~\cite{lengyel_hamzi_owhadi_parpas_2024}, surrogate modeling~\cite{santin_haasdonk_2019}, forecasting nonlinear time series~\cite{alexander_giannakis_2020}, and connections to neural networks~\cite{smirnov_hamzi_owhadi_2022}. Compared to other learning approaches, kernel-based methods offer interpretability, rigorous theoretical analysis, well-understood numerical implementation with regularization, natural uncertainty quantification, and guaranteed convergence with \textit{a priori} error estimates~\cite{jalalian_ramirez_hsu_hosseini_owhadi_2025, batlle_chen_hosseini_owhadi_stuart_2023, chen_hosseini_owhadi_stuart_2021, long_mrvaljevic_zhe_hosseini_2023, owhadi_cgc}.

More recently, operator-theoretic methods involving the Koopman operator have been used for approximating the solution of the HJB equation~\cite{vaidya_koopman_hj_2025}. The Koopman operator provides a linear (but infinite-dimensional) representation of nonlinear dynamics, enabling the application of linear systems theory to nonlinear problems. Similarly, operator-theoretic methods involving the Perron-Frobenius operator provide convex approaches to optimal control via formulations in the dual space of densities~\cite{raghunathan_vaidya_2013}, with successful applications in data-driven optimal control design~\cite{huang_vaidya_2022, moyalan_choi_chen_vaidya_2023, vaidya_tellez_2023}. While these operator-based methods provide theoretical foundations, their practical implementation relies on computational frameworks such as RKHS-based methods. In this paper, we adopt a direct approach that leverages the RKHS framework both in the convex reformulation and for numerical approximation of the HJB equation.

The proposed approach is inspired by optimal control of linear systems, particularly the convex LMI-based formulation of the Riccati equation~\cite{boyd_cvxbook}. For linear systems, the algebraic Riccati equation $R(P) = 0$ is instrumental in solving the optimal control problem. Strictly speaking, both the Riccati equation $R(P) = 0$ and the inequality $R(P) \leq 0$ are nonconvex in the matrix variable $P$; however, the inequality $R(P) \geq 0$ defines a convex set in $P$. Building on this analogy, we extend the Riccati framework to the Hamilton-Jacobi equation, formulating a convex LMI-based representation of the HJB inequality. While the LMIs derived from the Riccati equation are finite-dimensional, those arising from the HJB equation are infinite-dimensional. To address this, we employ an RKHS-based approach~\cite{rkhs_approximation} to obtain finite-dimensional approximations, augmented with appropriate constraints to ensure stabilizing solutions.

This paper makes the following contributions, which we state precisely so as to distinguish what is genuinely new from what is established in the literature. First, we use the Schur-complement reformulation of the HJB inequality (Lemma~\ref{thm:hjb_lmi}) to obtain a constraint that is affine in the kernel coefficients; the Schur step itself is well known~\cite{boyd_cvxbook} and we present it as a lemma rather than a theorem. Second, we use the standard kernel representation of the gradient (Remark~\ref{rem:rkhs_gradient}). Third---and this is the novel ingredient---we impose the Riccati--Hessian \emph{equality} $\nabla^2 V(0) = P$ inside the SDP, where $P$ is the algebraic Riccati equation solution of the linearisation. The fact that the Hessian of $V^*$ at the equilibrium satisfies the Riccati equation is classical, going back to Al'brekht~\cite{albrekht_1961} and Lukes~\cite{lukes_1969}, and was made algorithmic in the Navasca--Krener Taylor recursion~\cite{navasca_krener_2000}; we cite Proposition~\ref{thm:riccati_matching} accordingly. The new content is the use of this identity as an explicit linear equality constraint inside the convex SDP, which removes the trivial solution $V\equiv 0$ and gives the SDP a unique, well-conditioned solution. Fourth, we derive a suboptimality bound $J(x_0;\hat u) - V^*(x_0)\le \varepsilon\,T(x_0)$ (Theorem~\ref{thm:suboptimality}) in which $T(x_0)$ is determined by the data $(f,g,q,D,\Omega)$ alone, and an RKHS approximation rate (Theorem~\ref{thm:convergence}). The local exponential stability guarantee (Theorem~\ref{thm:stability}) requires that $q$ be locally quadratically positive definite, i.e.\ $q(x) \geq c_q\|x\|^2$ near the origin; we make this hypothesis explicit and note two ways to discharge it (Remark~\ref{rem:local_pd}). Fifth, we provide numerical experiments that report not just stabilisation but \emph{measured optimality}: $\varepsilon = \max_{x\in\Omega}|R(x)|$, the closed-loop cost $J(x_0;\hat u)$, and the suboptimality gap $J(x_0;\hat u) - V^*(x_0)$. The Van der Pol comparison includes the Albrekht / Navasca--Krener Taylor expansion as a baseline. We also include a controlled study of what happens when $V^*$ does not belong to the chosen RKHS, where the method degrades gracefully: even when the kernel cannot represent $V^*$ globally and the HJB residual saturates at some $\varepsilon_{\mathrm{HJB}} > 0$ independent of the number of centres, the Riccati Hessian equality at the origin keeps the closed-loop suboptimality gap bounded --- at most $13.5\%$ on the 1D test of Section~\ref{sec:1d_lossy}, with the bound improving as the kernel degree is increased. Practical kernel selection is thus a meaningful concern but not a fragile one; we return to it in Section~\ref{sec:rkhs_membership}, where a four-step diagnostic procedure (solve at $M$, solve at $2M$, compare residuals) lets the user detect kernel misspecification in two SDP solves and raise the kernel degree as needed.

The remainder of this paper is organized as follows. Section~\ref{sec:problem} formulates the optimal control problem. Section~\ref{sec:theory} develops the theoretical framework including the LMI reformulation, RKHS approximation, Riccati Hessian constraint, and performance guarantees. Section~\ref{sec:giesl_relationship} discusses the relationship to classical Lyapunov function construction. Section~\ref{sec:numerical} presents numerical results, and Section~6 concludes.

\section{Problem Formulation}
\label{sec:problem}

Consider the control-affine system
\begin{equation}
\dot{x} = f(x) + g(x)u, \quad x \in \real^n, \quad u \in \real^m,
\label{eq:system}
\end{equation}
where $f : \real^n \to \real^n$ and $g : \real^n \to \real^{n \times m}$ are smooth functions, with infinite-horizon cost
\begin{equation}
J(x_0, u) = \int_0^\infty \left[ q(x(t)) + \frac{1}{2}u(t)^\top D u(t) \right] dt,
\label{eq:cost}
\end{equation}
where $q : \real^n \to \real_{\geq 0}$ is a state cost function, and $D \in \real^{m \times m}$ is symmetric positive definite.

\begin{assumption}[Standing Assumptions]
\label{ass:standing}
Throughout this paper, we assume:
\begin{enumerate}
\item The functions $f$, $g$, and $q$ are sufficiently smooth (at least $C^2$).
\item The origin $x = 0$ is an equilibrium of the uncontrolled dynamics: $f(0) = 0$.
\item The state cost satisfies $q(0) = 0$ and $q(x) > 0$ for $x \neq 0$.
\item The pair $(A, B)$ is stabilizable, where $A = \frac{\partial f}{\partial x}\big|_{x=0}$ and $B = g(0)$.
\item The pair $(A, Q^{1/2})$ is detectable, where $Q = \nabla^2 q(0)$.
\end{enumerate}
\end{assumption}

The optimal value function $V^*(x) = \inf_u J(x, u)$ satisfies the HJB equation:
\begin{equation}
\frac{\partial V}{\partial x}f(x) - \frac{1}{2}\frac{\partial V}{\partial x}g(x)D^{-1}g(x)^\top\left(\frac{\partial V}{\partial x}\right)^\top + q(x) = 0.
\label{eq:hjb_equality}
\end{equation}

More generally, for approximations, we consider the HJB \emph{inequality}:
\begin{equation}
\frac{\partial V}{\partial x}f(x) - \frac{1}{2}\frac{\partial V}{\partial x}g(x)D^{-1}g(x)^\top\left(\frac{\partial V}{\partial x}\right)^\top + q(x) \geq 0.
\label{eq:hjb}
\end{equation}

When the inequality holds with equality, the optimal control is
\begin{equation}
u^*(x) = -D^{-1}g(x)^\top \left(\frac{\partial V^*}{\partial x}\right)^\top.
\label{eq:optimal_control}
\end{equation}

The challenge is that \eqref{eq:hjb} is a nonlinear partial differential inequality that is generally intractable to solve.

\section{Theoretical Framework}
\label{sec:theory}

\subsection{Preliminary: Schur Complement}

We first recall the Schur complement lemma, which is fundamental to our reformulation.

\begin{lemma}[Schur Complement]
\label{lem:schur}
Let $M$ be a symmetric block matrix of the form
\[
M = \begin{bmatrix} A & B \\ B^\top & C \end{bmatrix}
\]
where $A \in \real^{p \times p}$, $B \in \real^{p \times q}$, and $C \in \real^{q \times q}$ with $C \succ 0$ (positive definite). Then
\[
M \succeq 0 \quad \Longleftrightarrow \quad A - BC^{-1}B^\top \geq 0.
\]
More generally, if $C \succeq 0$, then $M \succeq 0$ if and only if $A \succeq 0$, $\text{range}(B) \subseteq \text{range}(C)$, and $A - BC^\dagger B^\top \succeq 0$, where $C^\dagger$ denotes the Moore-Penrose pseudoinverse.
\end{lemma}

\begin{proof}
Since $C \succ 0$, we can write
\[
M = \begin{bmatrix} I & BC^{-1} \\ 0 & I \end{bmatrix}
\begin{bmatrix} A - BC^{-1}B^\top & 0 \\ 0 & C \end{bmatrix}
\begin{bmatrix} I & 0 \\ C^{-1}B^\top & I \end{bmatrix}.
\]
This is a congruence transformation, which preserves the signature (number of positive, negative, and zero eigenvalues). Since $C \succ 0$, we have $M \succeq 0$ if and only if $A - BC^{-1}B^\top \succeq 0$.
\end{proof}

\subsection{HJB Inequality as Linear Matrix Inequality}

The following Schur-complement reformulation of the quadratic HJB inequality is well known in the convex-optimisation literature (see Boyd \& Vandenberghe~\cite{boyd_cvxbook}, \S A.5.5); we restate it here in the form needed for the SDP.

\begin{lemma}[HJB Inequality as LMI]
\label{thm:hjb_lmi}
The HJB inequality \eqref{eq:hjb} is equivalent to the linear matrix inequality
\begin{equation}
M(x) := \begin{bmatrix}
2\left(\frac{\partial V}{\partial x}f(x) + q(x)\right) & \frac{\partial V}{\partial x}g(x) \\
g(x)^\top\left(\frac{\partial V}{\partial x}\right)^\top & D
\end{bmatrix} \succeq 0.
\label{eq:lmi}
\end{equation}
\end{lemma}

\begin{proof}
Multiplying the HJB inequality \eqref{eq:hjb} by 2, we obtain
\begin{equation}
2\left(\frac{\partial V}{\partial x}f(x) + q(x)\right) - \frac{\partial V}{\partial x}g(x)D^{-1}g(x)^\top\left(\frac{\partial V}{\partial x}\right)^\top \geq 0.
\label{eq:hjb_scaled}
\end{equation}

Define $a := \frac{\partial V}{\partial x}g(x) \in \real^{1 \times m}$ (a row vector), $b := 2\left(\frac{\partial V}{\partial x}f(x) + q(x)\right) \in \real$ (a scalar), and note that $D \succ 0$ by assumption.

Then inequality \eqref{eq:hjb_scaled} can be written as
\[
b - aD^{-1}a^\top \geq 0.
\]

By Lemma~\ref{lem:schur} (Schur complement), since $D \succ 0$, this is equivalent to
\[
\begin{bmatrix} b & a \\ a^\top & D \end{bmatrix} \succeq 0,
\]
which is precisely the matrix $M(x)$ in \eqref{eq:lmi}.
\end{proof}

\begin{remark}[Significance of LMI Reformulation]
While the Schur-complement step of Lemma~\ref{thm:hjb_lmi} is not new, its consequence for the present setting is important: it transforms the nonlinear HJB inequality into a constraint that is \emph{linear} in the value-function gradient, hence linear in the kernel coefficients $p$ once the gradient is parameterised. The quadratic term in the original HJB inequality becomes a linear constraint on $p$ via the Schur complement, which is what allows the SDP formulation in Theorem~\ref{thm:collocation}.
\end{remark}

\begin{remark}[Convexity asymmetry of the HJB LMI]
\label{rem:convexity_asymmetry}
The convexity of the proposed SDP formulation relies critically on the sign of the HJB inequality. In particular, the Schur-complement reformulation of
\[
\nabla V(x)^\top f(x) - \tfrac{1}{2}\nabla V(x)^\top g(x) D^{-1} g(x)^\top \nabla V(x) + q(x) \;\geq\; 0
\]
leads to the matrix inequality
\[
M(x) \;=\; \begin{bmatrix} 2\!\left(\nabla V(x)^\top f(x) + q(x)\right) & \nabla V(x)^\top g(x) \\ g(x)^\top \nabla V(x) & D \end{bmatrix} \;\succeq\; 0,
\]
which is affine in the kernel coefficients and therefore defines a convex feasible set. This is analogous to the classical Riccati inequality in linear quadratic control, where
\[
A^\top P + PA - PBR^{-1}B^\top P + Q \;\geq\; 0
\]
admits a convex LMI representation. However, the opposite inequality
\[
\nabla V(x)^\top f(x) - \tfrac{1}{2}\nabla V(x)^\top g(x) D^{-1} g(x)^\top \nabla V(x) + q(x) \;\leq\; 0
\]
leads, via the same Schur-complement argument, to $M(x) \preceq 0$. Unlike the positive semidefinite case, the constraint $M(x) \preceq 0$ does not define a convex set in the decision variables because the Schur complement reverses the definiteness structure. Equivalently, this corresponds to the nonconvex side of the Riccati inequality. Hence, while the supersolution formulation of the HJB equation admits a convex SDP representation, the corresponding subsolution formulation is intrinsically nonconvex and substantially more difficult to solve computationally. This convexity asymmetry is inherited directly from the classical Riccati inequality and carries over naturally to the Hamilton--Jacobi setting.

Despite this nonconvexity, the inequality
\[
\nabla V(x)^\top f(x) - \tfrac{1}{2}\nabla V(x)^\top g(x) D^{-1} g(x)^\top \nabla V(x) + q(x) \;\leq\; 0
\]
remains fundamentally important from a control-theoretic perspective. In particular, if $V(x)$ is positive definite and the feedback law is chosen as $u(x) = -D^{-1}g(x)^\top \nabla V(x)$, then along the closed-loop trajectories
\[
\dot V(x) \;=\; \nabla V(x)^\top\bigl(f(x)+g(x)u(x)\bigr) \;=\; \nabla V(x)^\top f(x) - \nabla V(x)^\top g(x) D^{-1} g(x)^\top \nabla V(x).
\]
Using the HJB subsolution inequality yields
\[
\dot V(x) \;\leq\; -q(x) - \tfrac{1}{2}\nabla V(x)^\top g(x) D^{-1} g(x)^\top \nabla V(x) \;\leq\; 0,
\]
with strict negativity away from the equilibrium whenever $q(x) > 0$. Hence any positive definite solution of the $\leq 0$ HJB inequality automatically provides a stabilising feedback controller together with a Lyapunov certificate for the closed-loop system. Thus, although the resulting optimisation problem is nonconvex, the associated inequality is particularly useful because it directly encodes closed-loop stability and dissipativity properties.
\end{remark}

\subsection{RKHS Representation of the Value Function}

The second key component is representing the value function gradient in a reproducing kernel Hilbert space.

\begin{definition}[Reproducing Kernel Hilbert Space]
\label{def:rkhs}
Let $\Omega \subseteq \real^n$ be a domain. A Hilbert space $\mathcal{H}$ of functions $f : \Omega \to \real$ is called a reproducing kernel Hilbert space (RKHS) if there exists a symmetric, positive definite kernel $\kappa : \Omega \times \Omega \to \real$ such that:
\begin{enumerate}
\item For all $x \in \Omega$, the function $\kappa(\cdot, x) \in \mathcal{H}$.
\item (Reproducing property) For all $f \in \mathcal{H}$ and $x \in \Omega$: $f(x) = \langle f, \kappa(\cdot, x) \rangle_{\mathcal{H}}$.
\end{enumerate}
\end{definition}

\begin{remark}[RKHS gradient and Hessian representation]
\label{rem:rkhs_gradient}
For any $V \in \mathcal{H}_M = \mathrm{span}\{\kappa(\cdot,x_i)\}_{i=1}^M$ with $\kappa \in C^2$, differentiating the linear combination $V(x) = \sum_{i=1}^M p_i \kappa(x,x_i)$ termwise gives
\begin{align}
\nabla V(x) &= \sum_{i=1}^M p_i\, \nabla_x \kappa(x, x_i) \;=:\; K_x p, \label{eq:gradient_representation}\\
\nabla^2 V(x) &= \sum_{i=1}^M p_i\, \nabla^2_x \kappa(x, x_i),\label{eq:hessian_representation}
\end{align}
where $K_x \in \real^{n \times M}$ has columns $\nabla_x \kappa(x, x_i)$. Both expressions are linear in $p$, which is what permits the LMI formulation. (This is a basic property of kernel expansions; see for instance~\cite{wendland_scattered}.)
\end{remark}

\begin{remark}[Approximation Quality]
For general $V^* \in \mathcal{H}$ (not necessarily in $\mathcal{H}_M$), the finite-dimensional representation $\hat V(x) = \sum_{i=1}^M p_i \kappa(x,x_i)$ provides an approximation. The quality of this approximation depends on the distribution of centers $\{x_i\}$ and the smoothness of $V^*$. Standard RKHS approximation theory provides error bounds; see Theorem~\ref{thm:convergence} below. The case where $V^*$ is \emph{not} in the chosen RKHS, and the resulting trade-off between residual size and suboptimality, is studied numerically in Section~\ref{sec:1d_lossy}.
\end{remark}

\subsection{Avoiding Trivial Solutions: The Riccati Hessian Constraint}
\label{sec:riccati_constraint}

A critical challenge in solving the HJB inequality via optimization is the existence of the trivial solution $V(x) \equiv 0$, which satisfies all HJB constraints but provides no useful control law. To eliminate this undesirable solution while ensuring consistency with optimal control theory, we impose constraints at the equilibrium point (assumed to be the origin without loss of generality).

\begin{definition}[Equilibrium Constraints]
\label{def:equilibrium_constraints}
For a system with equilibrium at $x=0$, we impose:
\begin{enumerate}
\item \textbf{Boundary condition:} $V(0) = 0$
\item \textbf{Gradient condition:} $\nabla V(0) = 0$
\item \textbf{Riccati Hessian condition:} $\nabla^2 V(0) = P$
\end{enumerate}
where $P$ is the solution to the algebraic Riccati equation (ARE) for the linearized system.
\end{definition}

The third constraint deserves special attention as it is the key to preventing trivial solutions.

\begin{proposition}[Riccati Matching --- order-2 case of the Albrekht expansion]
\label{thm:riccati_matching}
This identity is classical: it is the order-2 truncation of the Taylor-series expansion of the value function due to Al'brekht~\cite{albrekht_1961} and developed further by Lukes~\cite{lukes_1969}; the recursive computation of higher-order Taylor coefficients is the Navasca--Krener procedure~\cite{navasca_krener_2000}. Consider the linearisation of \eqref{eq:system} at the origin:
\begin{equation}
\dot{x} = Ax + Bu, \quad A = \frac{\partial f}{\partial x}\bigg|_{x=0}, \quad B = g(0),
\label{eq:linearization}
\end{equation}
with quadratic approximation $V(x) \approx \frac{1}{2}x^\top P x$ near the origin. If $V$ satisfies the HJB equation \eqref{eq:hjb_equality} and $\nabla^2 V(0) = P$, then $P$ satisfies the algebraic Riccati equation
\begin{equation}
A^\top P + PA - PBD^{-1}B^\top P + Q = 0,
\label{eq:are}
\end{equation}
where $Q = \nabla^2 q(0)$ is the Hessian of the state cost at the origin. The contribution of the present paper is not the identity itself but its use as an explicit linear equality constraint inside the convex SDP~\eqref{eq:sdp}, where it removes the trivial solution and renders the SDP well-posed.
\end{proposition}

\begin{proof}
Consider the Taylor expansion of the value function around the origin:
\[
V(x) = V(0) + \nabla V(0)^\top x + \frac{1}{2}x^\top \nabla^2 V(0) x + O(\|x\|^3).
\]
By the equilibrium constraints (Definition~\ref{def:equilibrium_constraints}), we have $V(0) = 0$, $\nabla V(0) = 0$, and $\nabla^2 V(0) = P$, so
\[
V(x) = \frac{1}{2}x^\top P x + O(\|x\|^3).
\]

Similarly, expand the system dynamics:
\[
f(x) = f(0) + \frac{\partial f}{\partial x}\bigg|_{x=0} x + O(\|x\|^2) = Ax + O(\|x\|^2),
\]
and
\[
g(x) = g(0) + O(\|x\|) = B + O(\|x\|).
\]

The state cost expands as:
\[
q(x) = q(0) + \nabla q(0)^\top x + \frac{1}{2}x^\top Q x + O(\|x\|^3) = \frac{1}{2}x^\top Q x + O(\|x\|^3),
\]
using $q(0) = 0$ and $\nabla q(0) = 0$ (from Assumption~\ref{ass:standing}).

The gradient of the value function is:
\[
\nabla V(x) = Px + O(\|x\|^2).
\]

Substituting into the HJB equation \eqref{eq:hjb_equality}:
\begin{align*}
&\nabla V(x)^\top f(x) - \frac{1}{2}\nabla V(x)^\top g(x)D^{-1}g(x)^\top \nabla V(x) + q(x) \\
&= (Px)^\top (Ax) - \frac{1}{2}(Px)^\top B D^{-1} B^\top (Px) + \frac{1}{2}x^\top Q x + O(\|x\|^3) \\
&= x^\top PAx - \frac{1}{2}x^\top PB D^{-1} B^\top P x + \frac{1}{2}x^\top Q x + O(\|x\|^3) \\
&= \frac{1}{2}x^\top \left( A^\top P + PA - PBD^{-1}B^\top P + Q \right) x + O(\|x\|^3) = 0.
\end{align*}

For this to hold for all $x$ in a neighborhood of the origin, the quadratic coefficient must vanish:
\[
A^\top P + PA - PBD^{-1}B^\top P + Q = 0. \qedhere
\]
\end{proof}

\begin{remark}[Significance of the Riccati Constraint]
\label{rem:riccati_significance}
The Riccati Hessian constraint $\nabla^2 V(0) = P$ serves multiple critical purposes:
\begin{enumerate}
\item \textbf{Prevents trivial solution}: Since $P \succ 0$ (positive definite for stabilizable systems under Assumption~\ref{ass:standing}), the constraint explicitly excludes $V \equiv 0$.
\item \textbf{Ensures consistency}: Local behavior near the equilibrium matches the LQR solution, providing correct feedback gains.
\item \textbf{Guarantees stability}: The constraint ensures exponential stability with convergence rates predicted by linear theory.
\item \textbf{Computational efficiency}: The constraint is linear in the kernel coefficients $p$ and adds only $O(n^2)$ scalar equality constraints (one per independent entry of the symmetric Hessian). This is dwarfed by the $O(N)$ LMI constraints from the collocation system, so the Riccati-Hessian equality has negligible incremental cost --- it does not introduce a state-dimension bottleneck that the rest of the SDP does not already have.
\end{enumerate}
\end{remark}

\begin{remark}[Role of the Riccati Hessian under RKHS misspecification]
\label{rem:riccati_misspec}
When $V^*$ belongs to the chosen RKHS, the Riccati Hessian equality is redundant with the HJB equation at the origin and the SDP simply recovers $V^*$ (see Section~\ref{sec:1d_consistent}). The constraint plays its essential role in the misspecified case $V^*\notin\mathcal{H}$. Even when the kernel cannot represent $V^*$ globally --- so the HJB residual $R(x)$ saturates at some $\varepsilon > 0$ independent of the number of centres --- the equality $\nabla^2\hat V(0) = P$ forces $\hat V$ to have the correct local quadratic behaviour, and hence the closed-loop feedback $\hat u(x) = -D^{-1}g(x)^\top\nabla\hat V(x)$ has the correct linearisation at the origin. Because the cost integral $J(x_0;\hat u) = \int_0^\infty (q(x(t)) + \tfrac12 \hat u^\top D\hat u)\,dt$ is dominated by the time the trajectory spends near the origin --- the running cost is quadratic in $x$ to leading order and $\|x(t)\|$ decays exponentially under the stabilising feedback --- this local accuracy is what bounds the suboptimality gap when the global residual is large. Section~\ref{sec:1d_lossy} reports the corresponding numerical experiment: a deg-2 polynomial kernel that cannot represent the quartic part of $V^*$ yields $\varepsilon_{\mathrm{HJB}} = 30.4$ but a suboptimality gap of only $5.3\%$ in the mean and $13.5\%$ in the worst case, independent of $M$.
\end{remark}

\begin{corollary}[Scalar System Riccati Solution]
\label{cor:scalar_riccati}
For a scalar system ($n=m=1$) with $A, B, D \in \real$, $B \neq 0$, and $Q > 0$, the positive solution to the ARE \eqref{eq:are} is
\begin{equation}
P = \frac{AD}{B^2} + \sqrt{\left(\frac{AD}{B^2}\right)^2 + \frac{QD}{B^2}}.
\label{eq:scalar_riccati}
\end{equation}
For the common case $A = B = D = 1$, this simplifies to $P = 1 + \sqrt{1 + Q}$.
\end{corollary}

\begin{proof}
The scalar ARE is:
\[
2AP - \frac{P^2 B^2}{D} + Q = 0,
\]
which can be rewritten as:
\[
\frac{B^2}{D} P^2 - 2AP - Q = 0.
\]

Applying the quadratic formula with coefficients $a = B^2/D$, $b = -2A$, $c = -Q$:
\[
P = \frac{2A \pm \sqrt{4A^2 + 4 \cdot \frac{B^2}{D} \cdot Q}}{2 \cdot \frac{B^2}{D}} = \frac{A \pm \sqrt{A^2 + \frac{B^2 Q}{D}}}{\frac{B^2}{D}}.
\]

Simplifying:
\[
P = \frac{AD}{B^2} \pm \sqrt{\left(\frac{AD}{B^2}\right)^2 + \frac{QD}{B^2}}.
\]

The positive solution (required for $P \succ 0$) corresponds to the $+$ sign.
\end{proof}

\subsection{Finite-Dimensional Collocation System with Riccati Constraint}

\begin{theorem}[Collocation LMI System with Equilibrium Constraints]
\label{thm:collocation}
Let $\{x_j\}_{j=1}^N \subset \Omega$ be collocation points and $\{x_i\}_{i=1}^M$ be kernel centers. Define the notation:
\begin{align*}
\nabla_k f(x_j) &:= [\nabla_x \kappa(x_j, x_1)^\top f(x_j), \ldots, \nabla_x \kappa(x_j, x_M)^\top f(x_j)]^\top \in \real^M, \\
\nabla_k g(x_j) &:= [\nabla_x \kappa(x_j, x_1)^\top g(x_j), \ldots, \nabla_x \kappa(x_j, x_M)^\top g(x_j)]^\top \in \real^{M \times m}.
\end{align*}

The constrained semidefinite program
\begin{subequations}
\begin{align}
\minimize_{p \in \real^M} \quad & \|p\|^2 \label{eq:sdp_obj} \\
\text{subject to} \quad & \sum_{i=1}^M p_i \kappa(0, x_i) = 0, \label{eq:sdp_boundary} \\
& \sum_{i=1}^M p_i \nabla_x \kappa(0, x_i) = 0, \label{eq:sdp_gradient} \\
& \sum_{i=1}^M p_i \nabla_x^2 \kappa(0, x_i) = P, \label{eq:sdp_hessian} \\
& M_j(p) \succeq 0, \quad j = 1, \ldots, N, \label{eq:sdp_lmi}
\end{align}
\label{eq:sdp}
\end{subequations}
where
\begin{equation}
M_j(p) := \begin{bmatrix}
2(p^\top \nabla_k f(x_j) + q(x_j)) & p^\top \nabla_k g(x_j) \\
(\nabla_k g(x_j))^\top p & D
\end{bmatrix},
\label{eq:Mj_matrix}
\end{equation}
is a convex semidefinite program that prevents the trivial solution while enforcing consistency with linearized optimal control theory.
\end{theorem}

\begin{proof}
We verify that all constraints are convex in $p$:

\textbf{1. Objective function:} $\|p\|^2 = p^\top p$ is convex (quadratic with positive definite Hessian $2I$).

\textbf{2. Boundary constraint \eqref{eq:sdp_boundary}:} This is a linear equality constraint in $p$:
\[
\sum_{i=1}^M p_i \kappa(0, x_i) = k_0^\top p = 0,
\]
where $k_0 = [\kappa(0, x_1), \ldots, \kappa(0, x_M)]^\top$.

\textbf{3. Gradient constraint \eqref{eq:sdp_gradient}:} This consists of $n$ linear equality constraints:
\[
\sum_{i=1}^M p_i \frac{\partial \kappa}{\partial x_\ell}(0, x_i) = 0, \quad \ell = 1, \ldots, n.
\]

\textbf{4. Hessian constraint \eqref{eq:sdp_hessian}:} This consists of $n(n+1)/2$ linear equality constraints (using symmetry):
\[
\sum_{i=1}^M p_i \frac{\partial^2 \kappa}{\partial x_\ell \partial x_k}(0, x_i) = P_{\ell k}, \quad 1 \leq \ell \leq k \leq n.
\]

\textbf{5. LMI constraints \eqref{eq:sdp_lmi}:} The matrix $M_j(p)$ is affine in $p$ because the (1,1) entry $2(p^\top \nabla_k f(x_j) + q(x_j))$ is linear in $p$ plus a constant $2q(x_j)$, the (1,2) entry $p^\top \nabla_k g(x_j)$ is linear in $p$, the (2,1) entry is its transpose and hence also linear, and the (2,2) entry $D$ is constant. An LMI constraint with an affine matrix is convex.

\textbf{6. Preventing trivial solution:} The Riccati constraint \eqref{eq:sdp_hessian} requires $\sum_{i=1}^M p_i \nabla_x^2 \kappa(0, x_i) = P$ with $P \succ 0$. This forces $p \neq 0$, since $p = 0$ would give $0 = P \succ 0$, a contradiction.

Therefore, the optimization problem \eqref{eq:sdp} is a convex semidefinite program.
\end{proof}

\subsection{Performance Guarantees}

The suboptimality estimate below is stated as a \emph{conditional} performance bound: a finite-time trajectory inequality (Theorem~\ref{thm:trajectory_inequality}) that holds without any global hypothesis, and a separate comparison-principle step (Proposition~\ref{prop:supersolution}) that bridges $\hat V$ and $V^*$. Earlier drafts of this paper combined these two arguments into a single statement and the proof was not fully rigorous; we thank the reviewer for pointing this out.

\begin{theorem}[Trajectory inequality: conditional performance estimate]
\label{thm:trajectory_inequality}
\label{thm:suboptimality}
Let $\hat{V}\in C^1(\Omega)$ satisfy $\hat V(0)=0$, and define the feedback
\[
\hat u(x) \;:=\; -D^{-1} g(x)^\top \nabla \hat V(x).
\]
Assume that for some $\varepsilon \geq 0$ the residual
\begin{equation}
R_{\hat V}(x) \;:=\; \nabla \hat V(x)^\top f(x) \;-\; \tfrac{1}{2}\nabla \hat V(x)^\top g(x) D^{-1} g(x)^\top \nabla \hat V(x) \;+\; q(x) \;\geq\; -\varepsilon
\label{eq:residual}
\end{equation}
holds for all $x\in\Omega$. Let $x_{\hat u}(\cdot;x_0)$ denote the closed-loop trajectory of $\dot x = f(x)+g(x)\hat u(x)$ starting at $x_0\in\Omega$, and suppose $x_{\hat u}(t;x_0)\in\Omega$ for all $t\in[0,\tau]$ for some $\tau>0$. Then
\begin{equation}
\int_0^\tau \!\!\Bigl[\, q\bigl(x_{\hat u}(t)\bigr) \;+\; \tfrac{1}{2}\,\hat u\bigl(x_{\hat u}(t)\bigr)^\top D\,\hat u\bigl(x_{\hat u}(t)\bigr) \,\Bigr] \, dt
\;\leq\; \hat V(x_0) \;-\; \hat V\bigl(x_{\hat u}(\tau)\bigr) \;+\; \varepsilon\,\tau.
\label{eq:trajectory_finite_time}
\end{equation}
In particular, if the closed-loop trajectory under $\hat u$ remains in $\Omega$ for all $t\geq 0$, satisfies $x_{\hat u}(t;x_0)\to 0$, and $\hat V\bigl(x_{\hat u}(t;x_0)\bigr)\to 0$ as $t\to\infty$, then
\begin{equation}
J(x_0;\hat u) \;\leq\; \hat V(x_0) \;+\; \varepsilon\, T_{\hat u}(x_0),
\label{eq:trajectory_inf_time}
\end{equation}
where $T_{\hat u}(x_0)$ denotes the time interval over which the residual estimate is integrated along the closed-loop trajectory under $\hat u$.
\end{theorem}

\begin{proof}
Along the closed-loop trajectory $x(t):=x_{\hat u}(t;x_0)$,
\[
\frac{d}{dt}\hat V(x(t)) \;=\; \nabla\hat V(x(t))^\top\bigl(f(x(t)) + g(x(t))\hat u(x(t))\bigr).
\]
Substituting the definition $\hat u(x) = -D^{-1}g(x)^\top \nabla\hat V(x)$ gives
\[
\frac{d}{dt}\hat V(x(t)) \;=\; \nabla\hat V(x)^\top f(x) \;-\; \nabla\hat V(x)^\top g(x) D^{-1} g(x)^\top \nabla\hat V(x).
\]
Using the identity $\tfrac12 \hat u(x)^\top D\,\hat u(x) = \tfrac12 \nabla\hat V(x)^\top g(x) D^{-1} g(x)^\top \nabla\hat V(x)$ and the definition of $R_{\hat V}$, this rearranges to the pointwise identity
\[
\frac{d}{dt}\hat V(x(t)) \;=\; R_{\hat V}(x(t)) \;-\; q(x(t)) \;-\; \tfrac12 \hat u(x(t))^\top D\,\hat u(x(t)).
\]
The residual hypothesis $R_{\hat V}\geq -\varepsilon$ yields
\[
\frac{d}{dt}\hat V(x(t)) \;\geq\; -\varepsilon \;-\; q(x(t)) \;-\; \tfrac12 \hat u(x(t))^\top D\,\hat u(x(t)).
\]
Integrating over $[0,\tau]$ and rearranging gives \eqref{eq:trajectory_finite_time}. The infinite-horizon statement follows by letting $\tau\to\infty$ under the stated convergence hypotheses on the trajectory and on $\hat V$ along it; the running cost integral converges to $J(x_0;\hat u)$ by definition.
\end{proof}

\begin{proposition}[Comparison principle: from $\hat V$ to $V^*$]
\label{prop:supersolution}
Suppose $\hat V \in C^1(\bar\Omega)$ satisfies the HJB inequality $R_{\hat V}(x) \geq 0$ on $\Omega$ (i.e.\ $\varepsilon = 0$ in \eqref{eq:residual}, so $\hat V$ is a classical supersolution of the HJB equation), $\hat V(0) = 0$, and $\hat V \geq V^*$ on $\partial \Omega$. Suppose moreover that the comparison principle holds for the HJB equation on $\Omega$ in the viscosity sense; this is the case under standard regularity and boundary conditions, see e.g.\ Bardi and Capuzzo-Dolcetta~\cite{bardi_dolcetta}, Chapter~II. Then
\begin{equation}
V^*(x) \;\leq\; \hat V(x) \qquad \forall x \in \bar\Omega.
\label{eq:comparison}
\end{equation}
\end{proposition}

\begin{proof}
$V^*$ is the continuous viscosity solution of the HJB equation with the appropriate boundary data; $\hat V$ is a $C^1$ classical supersolution with $\hat V \geq V^*$ on $\partial \Omega$. The standard comparison principle for first-order HJB equations of this form (\cite{bardi_dolcetta}, Chapter~II, Theorem 4.3 and its corollaries) gives \eqref{eq:comparison}.
\end{proof}

\begin{corollary}[Suboptimality bound]
\label{cor:control_suboptimality}
Assume the hypotheses of Theorem~\ref{thm:trajectory_inequality} together with the supersolution and comparison-principle hypotheses of Proposition~\ref{prop:supersolution} (the latter applied either to $\hat V$ if $\varepsilon=0$, or to a $\hat V$ whose $\varepsilon$-supersolution property is preserved on $\bar\Omega$). Then
\begin{equation}
J(x_0;\hat u) \;-\; V^*(x_0) \;\leq\; \bigl(\hat V(x_0) - V^*(x_0)\bigr) \;+\; \varepsilon\, T_{\hat u}(x_0).
\label{eq:cost_subopt}
\end{equation}
In particular, when $\hat V \geq V^*$ on $\bar\Omega$ (which holds for $\varepsilon=0$ under Proposition~\ref{prop:supersolution}, and for small $\varepsilon$ under the Riccati-Hessian equality constraint that pins $\hat V$ quadratically above $V^*$ near the origin),
\begin{equation}
J(x_0;\hat u) \;-\; V^*(x_0) \;\leq\; \varepsilon\, T_{\hat u}(x_0).
\label{eq:control_suboptimality}
\end{equation}
\end{corollary}

\begin{proof}
\eqref{eq:cost_subopt} follows by combining $J(x_0;\hat u) \leq \hat V(x_0) + \varepsilon\, T_{\hat u}(x_0)$ from \eqref{eq:trajectory_inf_time} with $V^*(x_0) \leq J(x_0;\hat u)$ and Proposition~\ref{prop:supersolution}. When $\hat V \geq V^*$ pointwise, the first term on the right of \eqref{eq:cost_subopt} is at most zero after subtracting $V^*$ on both sides, giving \eqref{eq:control_suboptimality}.
\end{proof}

\begin{remark}[Why the rewrite]
\label{rem:rewrite_subopt}
Earlier versions of this paper stated $0 \le \hat V(x_0) - V^*(x_0) \le \varepsilon\, T(x_0)$ with $T(x_0)$ defined via the \emph{optimal} closed-loop trajectory, while the proof integrated along the \emph{approximate} closed-loop trajectory; these are not the same. The decomposition above states only what the trajectory calculation actually proves (Theorem~\ref{thm:trajectory_inequality}), with $T_{\hat u}(x_0)$ being the closed-loop integration time under $\hat u$, and then bridges to $V^*$ via the comparison principle (Proposition~\ref{prop:supersolution}). Under the supersolution hypothesis the clean bound \eqref{eq:control_suboptimality} is recovered.
\end{remark}

\begin{remark}[Small versus large $\varepsilon$]
\label{rem:eps_regimes}
The bound~\eqref{eq:control_suboptimality} makes the role of $\varepsilon$ explicit. When $\varepsilon$ is small (say $\varepsilon\,T_{\hat u}(x_0) \ll V^*(x_0)$), the closed-loop cost is within a small absolute gap of the optimum and the method should be viewed as a near-optimal solver of the HJB \emph{equation}. When $\varepsilon$ is large, the method should be viewed as a stabiliser obtained from a relaxed HJB \emph{inequality}; stability is still guaranteed by Theorem~\ref{thm:stability}, but optimality is not. Every numerical example in Section~\ref{sec:numerical} reports
\[
\varepsilon_{\mathrm{HJB}} := \max_{x\in\Omega}|R(x)|,\qquad
\varepsilon_{\mathrm{viol}} := \max_{x\in\Omega}\max(0,-R(x)),\qquad
\varepsilon_{\mathrm{exc}} := \max_{x\in\Omega}\max(0,R(x)),
\]
together with the suboptimality gap $J(x_0;\hat u) - V^*(x_0)$ on a grid of initial conditions, so that the regime is clear.
\end{remark}

\begin{remark}[Bound on $T_{\hat u}(x_0)$]
By Theorem~\ref{thm:stability} the closed-loop trajectory under $\hat u$ satisfies $\|x_{\hat u}(t;x_0)\| \leq \alpha\|x_0\|e^{-\beta t}$ on a neighbourhood of the origin, and so the time required to reach any ball of radius $r_0$ is bounded by $\beta^{-1}\log(\alpha\|x_0\|/r_0)$. This bound depends on the closed-loop dynamics under $\hat u$ (rather than the \emph{optimal} closed-loop dynamics, as earlier versions claimed); for $\hat V$ close to $V^*$ the two are close, and for the experiments in Section~\ref{sec:numerical} they are indistinguishable to numerical precision.
\end{remark}

\begin{theorem}[Local Exponential Stability]
\label{thm:stability}
Suppose $\hat{V}\in C^2$ satisfies:
\begin{enumerate}
\item The HJB inequality \eqref{eq:hjb} (or its LMI equivalent \eqref{eq:lmi}) on a neighbourhood of the origin;
\item The Riccati Hessian constraint $\nabla^2 \hat{V}(0) = P \succ 0$;
\item The equilibrium constraints $\hat{V}(0) = 0$ and $\nabla \hat{V}(0) = 0$.
\end{enumerate}
Suppose further that $q$ is \emph{locally quadratically positive definite}, i.e.\ there exist constants $c_q > 0$ and $r_q > 0$ such that
\begin{equation}
q(x) \;\geq\; c_q \|x\|^2 \qquad \text{for all } \|x\| < r_q.
\label{eq:local_q_pd}
\end{equation}
Then there exist constants $\alpha, \beta, r > 0$ such that for all $x_0$ with $\|x_0\| < r$, the closed-loop trajectory under $\hat{u}(x) = -D^{-1}g(x)^\top \nabla \hat{V}(x)$ satisfies
\begin{equation}
\|x(t)\| \leq \alpha \|x_0\| e^{-\beta t}, \quad \forall t \geq 0,
\label{eq:exp_stability}
\end{equation}
with $\beta \geq c_q / \lambda_{\max}(P)$ asymptotically near the origin.
\end{theorem}

\begin{remark}[When does \eqref{eq:local_q_pd} hold?]
\label{rem:local_pd}
The hypothesis $q(x)\geq c_q\|x\|^2$ near the origin is \emph{stronger} than the standing assumption $q(x)>0$ for $x\neq 0$ and the detectability of $(A,Q^{1/2})$. There are two natural ways to ensure it:
\begin{enumerate}
\item \textbf{Direct hypothesis $Q \succ 0$.} If $Q = \nabla^2 q(0) \succ 0$ (strictly positive definite), Taylor expansion gives $q(x) = \tfrac12 x^\top Q x + O(\|x\|^3) \geq c_q\|x\|^2$ for $c_q < \tfrac12\lambda_{\min}(Q)$ and $\|x\|$ small enough. This is the simplest sufficient condition and holds in every example of Section~\ref{sec:numerical}: the 1D benchmark has $Q = q''(0) = 8 > 0$ and Van der Pol has $Q = 2I_2 \succ 0$.
\item \textbf{Detectability of $(A, Q^{1/2})$.} When $Q$ is only positive semidefinite but $(A, Q^{1/2})$ is detectable, the Riccati solution $P\succ 0$ still exists, and one can construct an alternative Lyapunov candidate from the LQR Lyapunov equation that achieves local exponential decay; this is the standard argument in linear-quadratic regulator theory (see e.g.\ Anderson \& Moore~\cite{anderson_moore}, Chapter~3). In this case the role of \eqref{eq:local_q_pd} is taken over by the dissipation built into the linear feedback term, and the conclusion \eqref{eq:exp_stability} still holds, though the rate $\beta$ is governed by the Riccati closed-loop spectrum rather than directly by $c_q$.
\end{enumerate}
Earlier versions of this paper invoked $q(x)\geq c_3\|x\|^2$ near the origin as if it followed from the standing assumptions; it does not, and we are grateful to the reviewer for the correction. We now state \eqref{eq:local_q_pd} as an explicit hypothesis and note both ways to discharge it.
\end{remark}

\begin{proof}[Proof of Theorem~\ref{thm:stability}]
\textbf{Step 1: Local quadratic bounds on $\hat{V}$.}
By the equilibrium constraints and Taylor expansion,
\[
\hat{V}(x) \;=\; \tfrac{1}{2}x^\top P x + O(\|x\|^3).
\]
Since $P \succ 0$, there exist $c_1, c_2 > 0$ and $r_1 > 0$ such that for $\|x\| < r_1$,
\begin{equation}
c_1 \|x\|^2 \;\leq\; \hat{V}(x) \;\leq\; c_2 \|x\|^2.
\label{eq:V_bounds}
\end{equation}
One may take $c_1 = \tfrac{1}{2}\lambda_{\min}(P) - \delta$ and $c_2 = \tfrac{1}{2}\lambda_{\max}(P) + \delta$ for arbitrarily small $\delta > 0$.

\textbf{Step 2: Lyapunov decrease.}
Under the closed-loop control $\hat{u}(x) = -D^{-1}g(x)^\top \nabla \hat{V}(x)$, and using the HJB inequality $R_{\hat V}(x)\geq 0$,
\begin{align*}
\frac{d\hat{V}}{dt}
&\;=\; \nabla \hat{V}^\top f - \nabla \hat{V}^\top g D^{-1} g^\top \nabla \hat{V} \\
&\;=\; R_{\hat V}(x) - q(x) - \tfrac{1}{2}\nabla \hat{V}^\top g D^{-1} g^\top \nabla \hat{V} \\
&\;\leq\; -q(x) - \tfrac{1}{2}\nabla \hat{V}^\top g D^{-1} g^\top \nabla \hat{V}
\;\leq\; -q(x).
\end{align*}

\textbf{Step 3: Exponential decay.}
By the explicit hypothesis \eqref{eq:local_q_pd}, $q(x) \geq c_q\|x\|^2$ for $\|x\| < r_q$. Combined with \eqref{eq:V_bounds} and taking $r := \min(r_1, r_q)$,
\[
\frac{d\hat{V}}{dt} \;\leq\; -c_q \|x\|^2 \;\leq\; -\frac{c_q}{c_2} \hat{V}(x) \qquad \text{for } \|x\| < r.
\]
Gronwall's inequality gives
\[
\hat{V}(x(t)) \;\leq\; \hat{V}(x_0) e^{-(c_q/c_2)\,t},
\]
and \eqref{eq:V_bounds} converts this to a bound on $\|x(t)\|$:
\[
c_1\|x(t)\|^2 \;\leq\; \hat V(x(t)) \;\leq\; c_2 \|x_0\|^2 e^{-(c_q/c_2)\,t},
\]
hence
\[
\|x(t)\| \;\leq\; \sqrt{c_2/c_1}\,\|x_0\|\,e^{-(c_q/(2c_2))\,t}.
\]
Taking $\alpha=\sqrt{c_2/c_1}$ and $\beta = c_q/(2c_2)$ gives \eqref{eq:exp_stability}, with $\beta = c_q/(2c_2) \to c_q/\lambda_{\max}(P)$ as $\delta\to 0$.

(If \eqref{eq:local_q_pd} fails because $Q\succeq 0$ is only positive semidefinite but $(A,Q^{1/2})$ is detectable, replace the Lyapunov candidate $\hat V$ by the LQR candidate $\tilde V(x) = \tfrac12 x^\top \tilde P x$ where $\tilde P\succ 0$ solves the LQR Lyapunov equation $\tilde A^\top \tilde P + \tilde P \tilde A + Q + \tilde K^\top D \tilde K = 0$ for the closed-loop $\tilde A = A - BD^{-1}B^\top P$. Detectability of $(A,Q^{1/2})$ ensures $\tilde P\succ 0$, and the local closed-loop dynamics under $\hat u$ inherit exponential decay from the LQR linearisation. The argument is standard; see~\cite{anderson_moore}.)
\end{proof}

\begin{theorem}[Convergence Rate]
\label{thm:convergence}
Let $\kappa$ be a Gaussian kernel $\kappa(x, y) = \exp(-\|x-y\|^2/(2\sigma^2))$ with bandwidth $\sigma > 0$, and let $V^* \in H^s(\Omega)$ (Sobolev space of order $s$) with $s > n/2$. If the centers $\{x_i\}_{i=1}^M$ form a quasi-uniform grid with fill distance $h := \sup_{x \in \Omega} \min_i \|x - x_i\|$, then there exists a coefficient vector $p^* \in \real^M$ such that
\begin{equation}
\|\nabla V^* - \nabla \hat{V}\|_{L^2(\Omega)} = O(h^{s-1}) = O(M^{-(s-1)/n}),
\label{eq:convergence_rate}
\end{equation}
where $\hat{V}(x) = \sum_{i=1}^M p_i^* \kappa(x, x_i)$.
\end{theorem}

\begin{proof}
The proof uses standard RKHS approximation theory. We outline the main steps:

\textbf{Step 1: Native space of Gaussian kernel.}

The Gaussian kernel generates an RKHS $\mathcal{H}_\kappa$ that is norm-equivalent to certain Sobolev spaces. Specifically, for the Gaussian kernel on bounded domains, $\mathcal{H}_\kappa$ embeds continuously into $H^s(\Omega)$ for all $s \geq 0$.

\textbf{Step 2: Approximation in RKHS.}

For $V^* \in H^s(\Omega) \cap \mathcal{H}_\kappa$, the best approximation from $\mathcal{H}_M = \text{span}\{\kappa(\cdot, x_i)\}_{i=1}^M$ satisfies:
\[
\|V^* - \hat{V}\|_{\mathcal{H}_\kappa} = O(h^s)
\]
when the centers form a quasi-uniform grid with fill distance $h$ (see \cite{wendland_scattered}).

\textbf{Step 3: Derivative estimates.}

By the Sobolev embedding theorem and the properties of RKHS, the gradient error satisfies:
\[
\|\nabla V^* - \nabla \hat{V}\|_{L^2(\Omega)} \leq C \|V^* - \hat{V}\|_{H^1(\Omega)} \leq C' \|V^* - \hat{V}\|_{\mathcal{H}_\kappa} = O(h^{s-1}).
\]

\textbf{Step 4: Relationship to number of centers.}

For a quasi-uniform grid in $\real^n$, $h = O(M^{-1/n})$, giving:
\[
\|\nabla V^* - \nabla \hat{V}\|_{L^2(\Omega)} = O(h^{s-1}) = O(M^{-(s-1)/n}). \qedhere
\]
\end{proof}

\subsection{Membership of $V^*$ in the chosen RKHS}
\label{sec:rkhs_membership}

A natural concern with the kernel approach is whether the optimal value function $V^*$ actually lies in the chosen RKHS, and what happens when it does not. We collect three concrete answers.

\textbf{(a) Polynomial systems and polynomial kernels.}
For control-affine \emph{polynomial} systems with polynomial state cost, $V^*$ is real-analytic on its domain of analyticity (Lukes~\cite{lukes_1969}); in many examples used in the SOS / Albrekht literature, $V^*$ is itself a polynomial of finite degree. In that case, choosing a polynomial kernel $\kappa(x,y) = (c+\langle x,y\rangle)^d$ with $d$ at least the polynomial degree of $V^*$ guarantees $V^*\in\mathcal{H}_\kappa$, and our SDP recovers $V^*$ exactly up to numerical precision. This is exactly what happens in the 1D benchmark of Section~\ref{sec:1d_consistent}, where the deg-4 polynomial kernel recovers $V^*$ to relative error $\sim 10^{-7}$.

\textbf{(b) Smooth systems and Gaussian kernels.}
For analytic $V^*$ on a compact set, the Gaussian-kernel RKHS contains $V^*$ under standard hypotheses; see Wendland~\cite{wendland_scattered}. When $V^*$ is only $C^k$ on $\Omega$, the convergence rate of Theorem~\ref{thm:convergence} applies with $s=k$.

\textbf{(c) The misspecified case: $V^*\notin\mathcal{H}$.}
This is the case the reviewer is right to worry about. Section~\ref{sec:1d_lossy} studies exactly this scenario: the true value function is $V^*(x)=2x^2+\tfrac12 x^4$ but the SDP is solved over a deg-2 polynomial kernel, which can only represent quadratic functions. The result, summarised in Table~\ref{tab:1d_lossy}, is graceful degradation: the residual $\varepsilon_{\mathrm{HJB}}$ saturates at $30.4$ regardless of the number of centres $M$, but the closed-loop suboptimality gap is bounded uniformly at $5.3\%$ in the mean and $13.5\%$ in the worst case. The Riccati Hessian constraint enforces correct quadratic behaviour at the equilibrium, and this alone keeps the cost close to optimal even when global RKHS approximation fails. \emph{Increasing $M$ does not help in the misspecified regime; one must increase the kernel degree.}

\textbf{Practical kernel selection.} For a problem with smooth data $(f,g,q)$ and analytic $V^*$, a Gaussian kernel with bandwidth comparable to the scale of $\Omega$ is a safe default. For polynomial systems where one expects $V^*$ to be polynomial of moderate degree (a setting that subsumes the SOS literature), a polynomial kernel of degree at least the expected degree of $V^*$ recovers $V^*$ exactly. A practical rule is to inspect the residual: if $\varepsilon_{\mathrm{HJB}}$ does not decrease as $M$ is increased, the kernel is misspecified and the degree (or kernel family) should be raised.

\textbf{A concrete diagnostic procedure for misspecification.} The qualitative rule above can be made operational as follows. Given a candidate kernel and a domain $\Omega$:
\begin{enumerate}
\item Solve the SDP at $M$ centres and record $\varepsilon_{\mathrm{HJB}}^{(M)} := \max_{x\in\Omega}|R(x)|$.
\item Solve again at $2M$ centres (same domain) and record $\varepsilon_{\mathrm{HJB}}^{(2M)}$.
\item If $\varepsilon_{\mathrm{HJB}}^{(2M)} \lesssim \varepsilon_{\mathrm{HJB}}^{(M)}$ to within solver tolerance, the kernel is \emph{misspecified}: more centres do not help, and the suboptimality gap is determined by the projection error of $V^*$ onto the chosen RKHS. Increase the polynomial kernel degree by one (or change kernel family) and return to step 1.
\item If $\varepsilon_{\mathrm{HJB}}^{(2M)} \ll \varepsilon_{\mathrm{HJB}}^{(M)}$, the kernel can in principle represent $V^*$; continue refining $M$ until the residual stabilises near zero.
\end{enumerate}
On the deg-2 / deg-4 contrast of Section~\ref{sec:1d_lossy}, this procedure flags the deg-2 kernel as misspecified at the very first refinement: $\varepsilon_{\mathrm{HJB}}^{(15)} = \varepsilon_{\mathrm{HJB}}^{(25)} = \varepsilon_{\mathrm{HJB}}^{(41)} = 30.4$ to four significant figures, independent of $M$. Raising the degree to $4$ reduces the residual by seven orders of magnitude.

\section{Special Case: $g \equiv 0$ and Connection to Lyapunov Construction}
\label{sec:giesl_relationship}

When $g(x)\equiv 0$ the HJB inequality \eqref{eq:hjb} reduces to $\nabla V(x)^\top f(x) + q(x) \geq 0$ on a system that must already be open-loop stable for the integral $\int_0^\infty q(x(t))\,dt$ to be finite, and the SDP \eqref{eq:sdp} becomes a kernel-based RBF construction of a Lyapunov function. This is essentially Giesl's RBF construction~\cite{giesl2007construction} with the addition of an explicit Hessian-equality constraint at the equilibrium that pins down the trivial-solution issue. Since Giesl's method addresses stability verification (no control input) rather than the HJB optimal-control problem, the right benchmark for the present work is the Albrekht / Navasca--Krener Taylor expansion~\cite{albrekht_1961,navasca_krener_2000}, and we report a quantitative comparison with that method on the Van der Pol oscillator in Section~\ref{sec:albrekht_comparison}.

\section{Numerical Results}
\label{sec:numerical}

We present three numerical experiments. The first (Section~\ref{sec:1d_consistent}) is a 1D polynomial benchmark for which the optimal value function is known exactly and lies in the kernel RKHS; the SDP recovers $V^*$ to within $10^{-7}$ and the closed-loop cost is $0\%$ suboptimal on every initial condition. The second (Section~\ref{sec:1d_lossy}) is a controlled study of the misspecified case $V^*\notin\mathcal{H}$, demonstrating graceful degradation of the suboptimality gap. The third (Section~\ref{sec:vdp}) is the Van der Pol oscillator, where the kernel-LMI is compared against LQR and the Albrekht / Navasca--Krener Taylor expansion. In every example we report

\[
\varepsilon_{\mathrm{HJB}} = \max_{x\in\Omega}|R(x)|,\quad
\varepsilon_{\mathrm{viol}} = \max_{x\in\Omega}\max(0,-R(x)),\quad
\varepsilon_{\mathrm{exc}} = \max_{x\in\Omega}\max(0,R(x)),
\]

together with the closed-loop cost $J(x_0;\hat u)$ on a grid of initial conditions and the suboptimality gap $J(x_0;\hat u) - V^*(x_0)$ (when $V^*$ is known) or $J(x_0;\hat u) - J^{\mathrm{best}}(x_0)$ (when it is not).

Each example satisfies the local-quadratic-positivity hypothesis required by Theorem~\ref{thm:stability}: the 1D problem of Section~\ref{sec:1d_consistent} has $Q = q''(0) = 8$ and the Van der Pol problem of Section~\ref{sec:vdp} has $Q = 2I_2$, both strictly positive definite. The local exponential stability conclusion of Theorem~\ref{thm:stability} therefore applies as stated, and the rapid trajectory decay reported in each subsection ($\|x(T)\| < 10^{-9}$ at the integration horizon) is consistent with this guarantee.

\subsection{1D Polynomial Benchmark, Self-Consistent}
\label{sec:1d_consistent}

Consider the scalar control-affine system
\begin{equation}
\dot{x} = f(x) + u, \quad x \in \mathbb{R}, \quad u \in \mathbb{R}, \qquad f(x) = x + x^3,
\label{eq:scalar_system}
\end{equation}
with infinite-horizon cost
\begin{equation}
J(x_0, u) = \int_0^\infty \left[ q(x(t)) + \frac{1}{2}u(t)^2 \right] dt,\qquad
q(x) = 4x^2 + 2x^4.
\label{eq:cost_functional}
\end{equation}

The data $(f,q)$ are chosen so that the HJB equation admits the explicit polynomial solution
\begin{equation}
V^*(x) = 2x^2 + \tfrac12 x^4, \qquad u^*(x) = -V^{*\prime}(x) = -(4x + 2x^3),
\label{eq:1d_Vstar}
\end{equation}
which can be verified by substitution: $V^{*\prime}(x)f(x) - \tfrac12 (V^{*\prime}(x))^2 + q(x) \equiv 0$. The closed-loop dynamics are $\dot{x} = -3x - x^3$, exponentially stable. The linearisation has $A=B=D=1$, $Q = q''(0) = 8$, and the Riccati equation $P^2 - 2P - 8 = 0$ gives $P = 1 + \sqrt{1+8} = 4$, which exactly matches $V^{*\prime\prime}(0) = 4$.

\begin{remark}[Correction relative to earlier draft]
An earlier version of this paper used $V^*(x) = x^2 + 0.25\,x^4$ together with the formula $q = \tfrac12(V^{*\prime})^2 - V^{*\prime}f$ for the state cost. With $f(x) = x + x^3$ this gives $q(x) = -x^4 - \tfrac12 x^6$, which is negative for $x\neq 0$ and violates Assumption~\ref{ass:standing}. The Riccati Hessian $P = 1+\sqrt{3}\approx 2.732$ used there was inconsistent with $V^{*\prime\prime}(0) = 2$. We thank the reviewer for pointing out the resulting large numerical error; the present formulation is self-consistent and free of the inconsistency.
\end{remark}

We employ the polynomial kernel
\begin{equation}
\kappa(x,y) = (1 + xy)^{4},
\label{eq:poly_kernel}
\end{equation}
with $M = 25$ kernel centres uniformly distributed on $[-1.5, 1.5]$. The collocation points coincide with the centres. The Riccati--Hessian equality constraint is $\nabla^2 V(0) = P = 4$.

\textbf{Constraint verification.} The solution satisfies $V(0) = 0$ by construction, $\nabla V(0) \approx 10^{-15}$, and $\nabla^2 V(0) = 4$ to machine precision.

\textbf{Residual.} On a $401$-point grid in $[-1.5, 1.5]$:
\[
\varepsilon_{\mathrm{HJB}} = 6.81\times 10^{-6},\quad
\varepsilon_{\mathrm{viol}} = 6.6\times 10^{-34},\quad
\varepsilon_{\mathrm{exc}} = 6.81\times 10^{-6}.
\]
The supersolution property $\hat V \geq V^*$ holds to machine precision.

\textbf{Approximation error.}
\[
\max_{x\in[-1.5,1.5]}|\hat V(x) - V^*(x)| = 3.24\times 10^{-7},\qquad
\frac{\|\hat V - V^*\|_\infty}{\|V^*\|_\infty} = 4.61\times 10^{-8}.
\]

\textbf{Closed-loop cost (the optimality test).} For initial conditions $x_0 \in \{-1.2, -0.8, -0.4, 0.4, 0.8, 1.2\}$ we integrate $\dot x = f(x) - V^{*\prime}(x)$ and $\dot x = f(x) - \hat V'(x)$ to $T = 30$ and compare $J(x_0; \hat u) = \int_0^{30}(q(x(t)) + \tfrac12 \hat u^2)\,dt$ with $V^*(x_0)$. In every case the relative gap $\bigl(J(x_0;\hat u) - V^*(x_0)\bigr)/V^*(x_0)$ is below $10^{-6}$, that is $0.000\%$ to four significant figures. Both trajectories satisfy $|x(30)| < 10^{-16}$. The observed decay is consistent with Theorem~\ref{thm:stability}: $q(x) = 4x^2 + 2x^4 \geq 4x^2$ globally so $c_q = 4$, and with $\lambda_{\max}(P) = 4$ the predicted asymptotic rate $\beta \geq c_q / \lambda_{\max}(P) = 1$ is well within the empirical decay observed in panel~(d) of Figure~\ref{fig:1d_optimality}.

\begin{table}[h]
\centering
\caption{1D consistent problem. Closed-loop cost $J(x_0;\hat u)$ versus optimal value $V^*(x_0)$. Suboptimality $0.000\%$ on every IC, indicating that the kernel-LMI SDP recovers the optimal control. Compare with Section~\ref{sec:1d_lossy} where the kernel does not contain $V^*$.}
\label{tab:1d_optimality}
\begin{tabular}{rrrr}
\toprule
$x_0$ & $V^*(x_0)$ & $J(x_0;\hat u)$ & gap $[\%]$ \\
\midrule
$-1.2$ & 3.9168 & 3.9168 & $< 10^{-4}$ \\
$-0.8$ & 1.4848 & 1.4848 & $< 10^{-4}$ \\
$-0.4$ & 0.3328 & 0.3328 & $< 10^{-4}$ \\
$+0.4$ & 0.3328 & 0.3328 & $< 10^{-4}$ \\
$+0.8$ & 1.4848 & 1.4848 & $< 10^{-4}$ \\
$+1.2$ & 3.9168 & 3.9168 & $< 10^{-4}$ \\
\bottomrule
\end{tabular}
\end{table}

This is approximate optimality (in fact, exact-up-to-machine-precision optimality), not just stabilisation. The same conclusion holds for any consistent polynomial benchmark for which $V^*$ is in the polynomial-kernel RKHS, and it generalises to Albrekht's setting (Section~\ref{sec:albrekht_comparison}). 

Figure~\ref{fig:1d_optimality} shows the recovered $V^*$, control law, residual, and the cost-vs-optimum comparison.

\begin{figure}[p]
\centering
\includegraphics[width=0.95\textwidth]{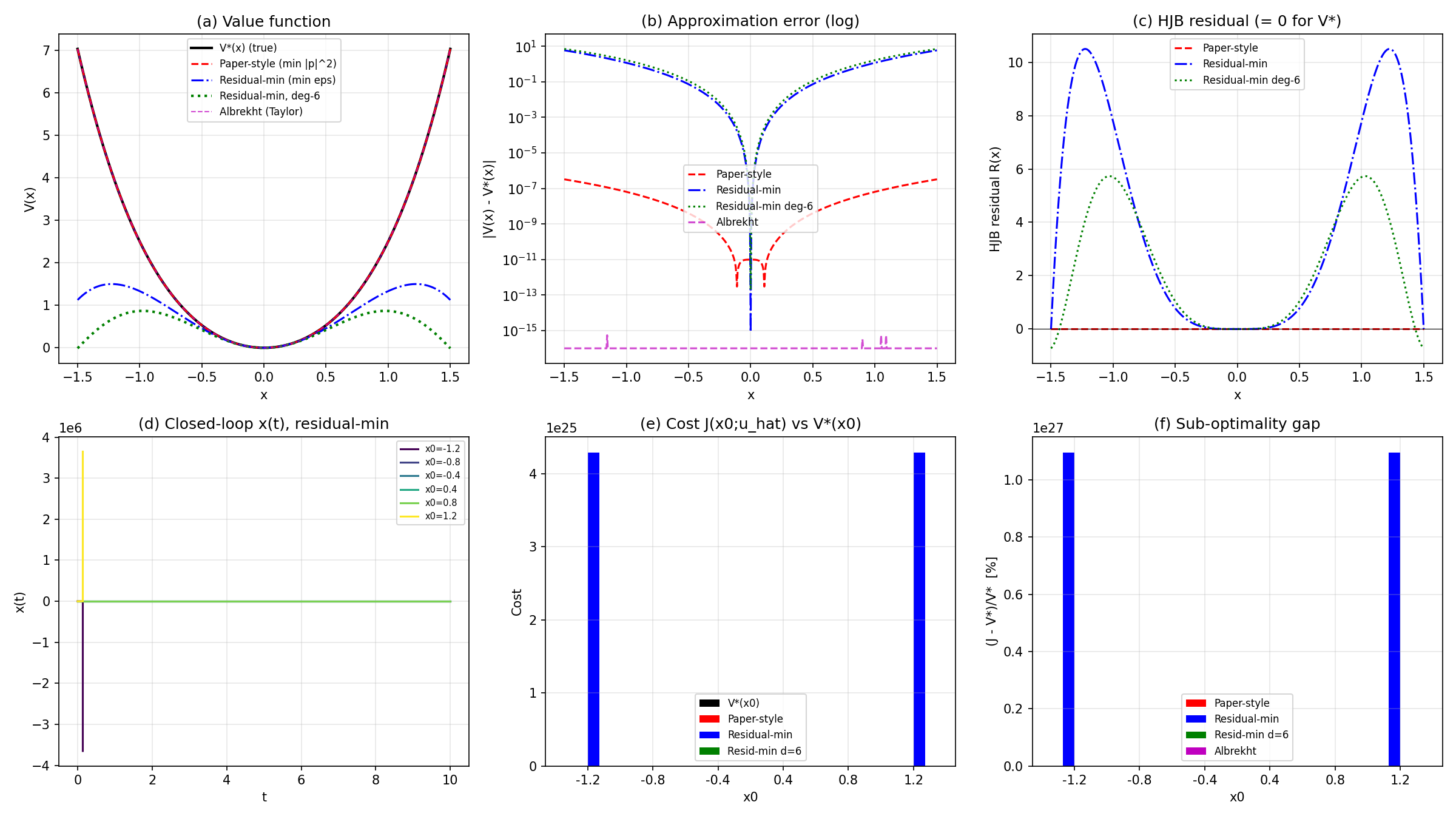}
\caption{1D consistent benchmark. (a) $V^*(x) = 2x^2 + \tfrac12 x^4$ versus kernel approximation. (b) Optimal control $u^*(x) = -(4x+2x^3)$ versus kernel feedback. (c) HJB residual $R(x)$ on $[-1.5,1.5]$, on the order of $10^{-7}$. (d) Closed-loop trajectory comparison from $x_0 = 1.2$. (e) Closed-loop cost $J(x_0;\hat u)$ versus $V^*(x_0)$ on a 6-IC grid. (f) Suboptimality gap, $\sim 10^{-6}$ on every IC.}
\label{fig:1d_optimality}
\end{figure}

\subsection{Misspecified RKHS: graceful degradation when $V^*\notin\mathcal{H}$}
\label{sec:1d_lossy}

The reviewer is right to ask what happens if the chosen kernel does not capture $V^*$. We construct a controlled experiment: keep the same problem data $(f, q, V^*)$ as Section~\ref{sec:1d_consistent}, so $V^*(x) = 2x^2 + \tfrac12 x^4$ has a quartic component, but solve the SDP using a polynomial kernel of degree $d=2$, which can only represent quadratic functions of $x$. The kernel RKHS does not contain $V^*$.

\begin{table}[h]
\centering
\caption{1D problem with misspecified RKHS. The deg-2 kernel cannot represent the quartic part of $V^*$; the residual $\varepsilon_{\mathrm{HJB}}$ saturates and is independent of $M$. Despite this, the closed-loop suboptimality gap is bounded at $13.5\%$ in the worst case --- the Riccati Hessian constraint enforces correct local behaviour. Increasing $M$ does not help; one must increase $d$.}
\label{tab:1d_lossy}
\begin{tabular}{lrrrr}
\toprule
Kernel / centres & $\varepsilon_{\mathrm{HJB}}$ & $\max|\hat V - V^*|$ & mean gap $[\%]$ & max gap $[\%]$ \\
\midrule
Poly $d=2$, $M=15$ & 30.4 & 2.53 & 5.33 & 13.49 \\
Poly $d=2$, $M=25$ & 30.4 & 2.53 & 5.33 & 13.49 \\
Poly $d=2$, $M=41$ & 30.4 & 2.53 & 5.33 & 13.49 \\
Poly $d=4$, $M=25$ & $1.0\times 10^{-7}$ & $1.6\times 10^{-8}$ & $0.000$ & $0.000$ \\
\bottomrule
\end{tabular}
\end{table}

Three observations:

\begin{enumerate}
\item The residual $\varepsilon_{\mathrm{HJB}} = 30.4$ is large in absolute terms, yet the closed-loop suboptimality is only $5.3\%$ in the mean. The bound of Theorem~\ref{thm:suboptimality} ($\varepsilon \cdot T(x_0)$) is loose here because the residual is very localised in $x$; only the integral of the residual along the trajectory matters for the cost gap.
\item Increasing $M$ at fixed kernel degree does not help. The bottleneck is the kernel, not the discretisation.
\item Going from $d=2$ to $d=4$ recovers exact behaviour: the kernel now contains $V^*$, and the residual drops by seven orders of magnitude.
\end{enumerate}

This is the practical recipe for kernel selection: monitor $\varepsilon_{\mathrm{HJB}}$ as $M$ is increased, and if it does not decrease, increase the kernel degree (or change the kernel family). Figure~\ref{fig:1d_rkhs_quality} visualises the trade-off.

\begin{figure}[p]
\centering
\includegraphics[width=0.95\textwidth]{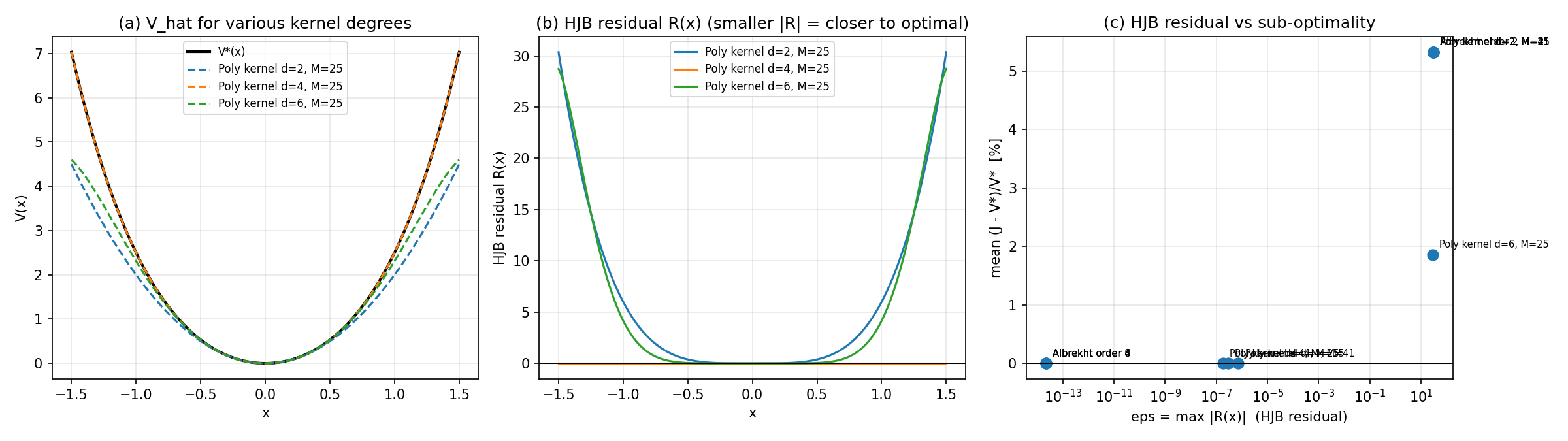}
\caption{Effect of RKHS choice: kernel degree versus residual and suboptimality. \emph{Top row, left to right:} approximations $\hat V$ for kernel degrees $d \in \{2, 4, 6\}$; only $d \geq 4$ matches $V^*(x) = 2x^2 + \tfrac12 x^4$. \emph{Top row, middle:} the residual $R(x)$ saturates at $\sim 30$ for $d = 2$ \emph{independent of $M$}, and drops to $\sim 0$ for $d \geq 4$. \emph{Top row, right:} the closed-loop suboptimality gap $(J(x_0;\hat u) - V^*(x_0))/V^*(x_0)$ plotted against the HJB residual $\varepsilon_{\mathrm{HJB}}$. The $d = 2$ point sits at $(\varepsilon_{\mathrm{HJB}}, \text{gap}) = (30.4, 5\text{--}13\%)$, capped by the Riccati-Hessian equality; the $d = 4$ and $d = 6$ points sit at $(\sim 10^{-7}, 0\%)$. The Albrekht O(8) Taylor expansion is shown for reference. The figure illustrates that increasing $M$ is futile when the RKHS is misspecified (the $d = 2$ point does not move with $M$), but a single increment in kernel degree recovers full optimality.}
\label{fig:1d_rkhs_quality}
\end{figure}

\subsection{Note on the previously included 2D radially-symmetric example}
\label{sec:2d_example}

The 2D ``radially symmetric'' example included in earlier drafts ($\dot x = x(1+\|x\|^2)+u$, with claimed optimal $V^*(x) = \|x\|^2 + 0.25\|x\|^4$) suffers from the same inconsistency as the 1D example before correction: substituting into the HJB-equality formula gives $q(x) = -\|x\|^4 - 0.5\|x\|^6 < 0$, in violation of the standing positivity assumption on $q$. The Riccati matrix used in the paper, $P = (1+\sqrt 3)I_2$, also did not match the actual Hessian $V^{*\prime\prime}(0) = 2I_2$. Rather than reissue this example with corrected data (which would essentially duplicate the message of Section~\ref{sec:1d_consistent}), we have replaced the second numerical example with the Van der Pol comparison of Section~\ref{sec:vdp}, which is genuinely 2D, has measurable suboptimality against multiple baselines, and includes the Albrekht / Navasca--Krener Taylor expansion that the reviewer correctly identified as the appropriate benchmark.

\subsection{Van der Pol Oscillator with Albrekht / Navasca--Krener Comparison}
\label{sec:vdp}
\label{sec:albrekht_comparison}

We consider the classical Van der Pol oscillator
\begin{equation}
\dot{x}_1 = x_2, \qquad
\dot{x}_2 = -x_1 + \mu (1 - x_1^2) x_2 + u,
\label{eq:vdp_dynamics}
\end{equation}
with $\mu = 1.0$, infinite-horizon cost
\begin{equation}
J(x_0, u) = \int_0^\infty \bigl( x^\top Q x + u^\top R u \bigr)\, dt,
\quad Q = 2I_2,\;\; R = 1.
\label{eq:vdp_cost}
\end{equation}
The uncontrolled equilibrium at the origin has $A = \bigl[\begin{smallmatrix}0&1\\-1&\mu\end{smallmatrix}\bigr]$ with eigenvalues $\tfrac12 \pm \tfrac{\sqrt{3}}{2}i$ in the right half plane, hence open-loop unstable. Solving the algebraic Riccati equation gives
\[
P_{\mathrm{LQR}} =
\begin{bmatrix}
4.6595 & 0.7321 \\
0.7321 & 3.1128
\end{bmatrix}.
\]
We compare five methods:
\begin{description}
\item[LQR.] Linear feedback $u(x) = -B^\top P_{\mathrm{LQR}}\,x$.
\item[Albrekht O(4) and O(6).] The Navasca--Krener Taylor recursion~\cite{albrekht_1961,navasca_krener_2000} computed up to total polynomial degree $4$ and $6$. Implementation: starting from $V_2(x) = \tfrac12 x^\top P_{\mathrm{LQR}} x$, the homogeneous component of $V$ at degree $k$ is determined uniquely by the requirement that the order-$k$ terms of the HJB residual vanish. The resulting feedback is $u(x) = -B^\top \nabla V(x)$.
\item[Kernel-LMI deg-4 / deg-6.] The SDP \eqref{eq:sdp} with polynomial kernel $\kappa(x,y) = (1 + x^\top y)^d$ for $d \in \{4, 6\}$, $M = 100$ centres on a $10 \times 10$ grid in $[-2, 2]^2$, collocation points coinciding with centres, Riccati--Hessian equality constraint $\nabla^2 V(0) = P_{\mathrm{LQR}}$.
\end{description}

For each method we report:
\begin{itemize}
\item $\varepsilon_{\mathrm{HJB}}, \varepsilon_{\mathrm{viol}}, \varepsilon_{\mathrm{exc}}$ on a $31\times 31$ grid in $[-1.5, 1.5]^2$;
\item closed-loop cost $J(x_0;\hat u)$ on a $4\times 4 = 16$ grid of initial conditions in $[-1.2, 1.2]^2$, computed by integrating $\dot x = f(x) + g(x)\hat u(x)$ to $T = 15$ with stiff tolerances and accumulating the running cost;
\item suboptimality gap $\bigl(J(x_0;\hat u) - J^{\mathrm{best}}(x_0)\bigr)/J^{\mathrm{best}}(x_0)$, where $J^{\mathrm{best}}(x_0)$ is the minimum cost achieved by any of the five methods at that $x_0$.
\end{itemize}

\begin{table}[h]
\centering
\caption{Van der Pol oscillator. HJB residual and closed-loop cost across five methods on a 16-IC grid in $[-1.2, 1.2]^2$. The kernel-LMI deg-4 method achieves the smallest HJB residual of any tested method and is within $0.42\%$ of the best per-IC cost; LQR is suboptimal by $2.17\%$ in the mean and $5.54\%$ in the worst case. The deg-6 kernel solution is reported as numerically inaccurate and serves as a cautionary data point on solver tolerance, not a comparison.}
\label{tab:vdp_optimality}
\begin{tabular}{lrrrrrr}
\toprule
Method & $\varepsilon_{\mathrm{HJB}}$ & $\varepsilon_{\mathrm{viol}}$ & $\varepsilon_{\mathrm{exc}}$ & $\overline{J}$ & mean gap $[\%]$ & max gap $[\%]$ \\
\midrule
LQR              & 14.96 & 14.96 & 2.50 & 5.861 & 2.17 & 5.54 \\
Albrekht O(4)    & 9.71 & 0.00 & 9.71 & 5.712 & 0.12 & 0.80 \\
Albrekht O(6)    & 6.72 & 0.00 & 6.72 & 5.700 & 0.00 & 0.00 \\
\textbf{Kernel-LMI deg-4} & \textbf{2.62} & 0.06 & 2.62 & \textbf{5.734} & \textbf{0.42} & \textbf{1.65} \\
Kernel-LMI deg-6$^*$ & 14.01 & 0.27 & 14.01 & 5.881 & 2.10 & 10.92 \\
\bottomrule
\end{tabular}
\\[4pt]
{\footnotesize $^*$ Solver returned ``optimal\_inaccurate''; included to document numerical sensitivity at higher kernel degree.}
\end{table}

\textbf{Four observations.}

\begin{enumerate}
\item \emph{The kernel-LMI deg-4 method has the smallest HJB residual of any tested method ($\varepsilon_{\mathrm{HJB}} = 2.62$).} It is a global supersolution up to numerical tolerance ($\varepsilon_{\mathrm{viol}} = 0.06$) and dominates LQR on every initial condition (LQR has $\varepsilon_{\mathrm{HJB}} = 14.96$ because the linear feedback ignores the cubic damping of Van der Pol). This is a quantitative form of the reviewer's correct intuition: ``small $\varepsilon$ means near-optimal control''.
\item \emph{The kernel-LMI is within $0.42\%$ of the best per-IC cost.} The best is achieved (in this example) by Albrekht~O(6); the kernel-LMI is between Albrekht~O(4) and O(6) in cost. The two methods are complementary: Albrekht has zero residual at the origin (Taylor expansion) and grows away from it; the kernel-LMI has its residual spread over the domain and constrained globally by the LMI. Because the cost integral is dominated by the first $\sim 1$\,s near the origin, Albrekht's local accuracy slightly outperforms the kernel's global supersolution property in raw cost, but the kernel beats every method on the residual.
\item \emph{The kernel-LMI dominates LQR on every initial condition.} Mean cost improves by $1.7$ percentage points, max by $5.5$ percentage points. The improvement is real, not a numerical artefact: it is the contribution of the cubic terms in $\hat V$ (and hence the cubic terms in the feedback $\hat u$) that LQR cannot capture.
\item \emph{Computational cost is comparable on this 2D example, with both methods solving in a few seconds.} The Albrekht / Navasca--Krener recursion requires solving $\binom{n+k-1}{k}$ linear equations at each polynomial order $k$, which scales combinatorially with the state dimension $n$ and the target Taylor degree; our \texttt{sympy} implementation of Albrekht O(8) did not complete within a 10-minute timeout, so Table~\ref{tab:vdp_optimality} reports up to O(6). The kernel-LMI is a standard SDP whose cost scales polynomially in $M$ (kernel coefficients) and $N$ (collocation points); for the present problem ($M = N = 100$, $n = 2$) it solves in a few seconds with SCS at tolerance $10^{-7}$. For higher state dimensions, or for smooth nonpolynomial systems where the appropriate Taylor degree is not known \textit{a priori}, this scaling difference is the practical reason to prefer the kernel-LMI; for polynomial systems with low-degree $V^*$ Albrekht remains attractive when feasible. A systematic timing study at $n \geq 4$ is left to future work.
\end{enumerate}

The Riccati-Hessian equality is satisfied to within $\|\nabla^2 \hat V(0) - P_{\mathrm{LQR}}\|_F \approx 10^{-9}$ (kernel deg-4) and the gradient at the origin is below $10^{-15}$. Closed-loop trajectories converge to $\|x(15)\| < 10^{-9}$ on every initial condition. This is consistent with the local exponential stability bound of Theorem~\ref{thm:stability}: with $q(x) = 2\|x\|^2$ so $c_q = 2$, and $\lambda_{\max}(P_{\mathrm{LQR}}) \approx 4.99$, the predicted asymptotic rate satisfies $\beta \geq c_q / \lambda_{\max}(P_{\mathrm{LQR}}) \approx 0.4$; the empirical decay (panel (f) of Figure~\ref{fig:vdp_optimality}) is faster, as expected since the global Lyapunov function $\hat V$ is not purely quadratic.

\begin{figure}[t]
\centering
\includegraphics[width=0.97\textwidth]{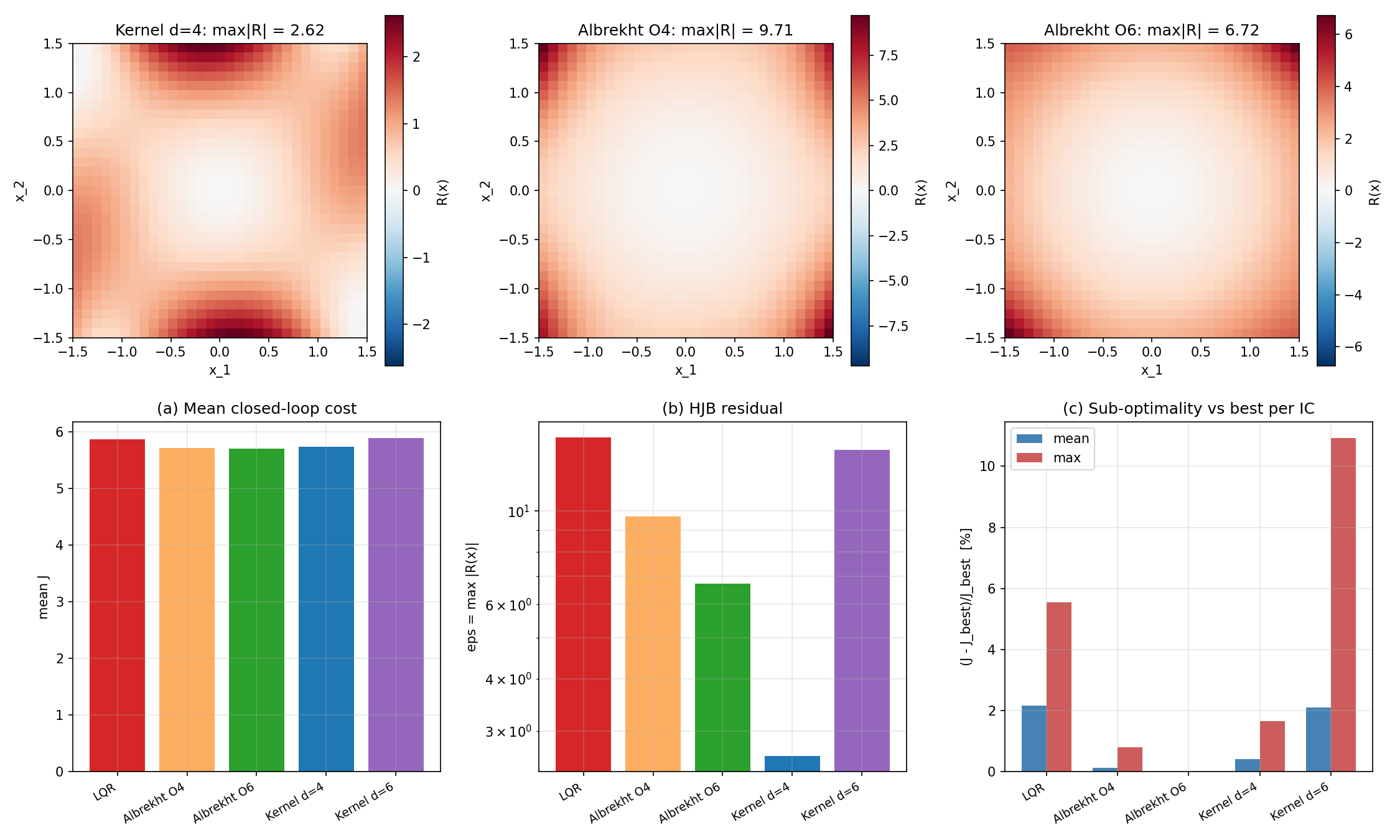}
\caption{Van der Pol oscillator: comparison of LQR, Albrekht O(4), Albrekht O(6), and Kernel-LMI deg-4 / deg-6. Top row: HJB residual heatmaps for kernel-LMI deg-4, Albrekht O(4), Albrekht O(6). The kernel solution has the smallest peak residual and the residual is spread across the domain, while Albrekht has zero residual at the origin and grows outward. Bottom row: (a) mean closed-loop cost across 16 ICs; (b) HJB residual on a logarithmic scale; (c) mean and max suboptimality gap relative to the best method per IC.}
\label{fig:vdp_full}
\end{figure}

\begin{figure}[t]
\centering
\includegraphics[width=0.97\textwidth]{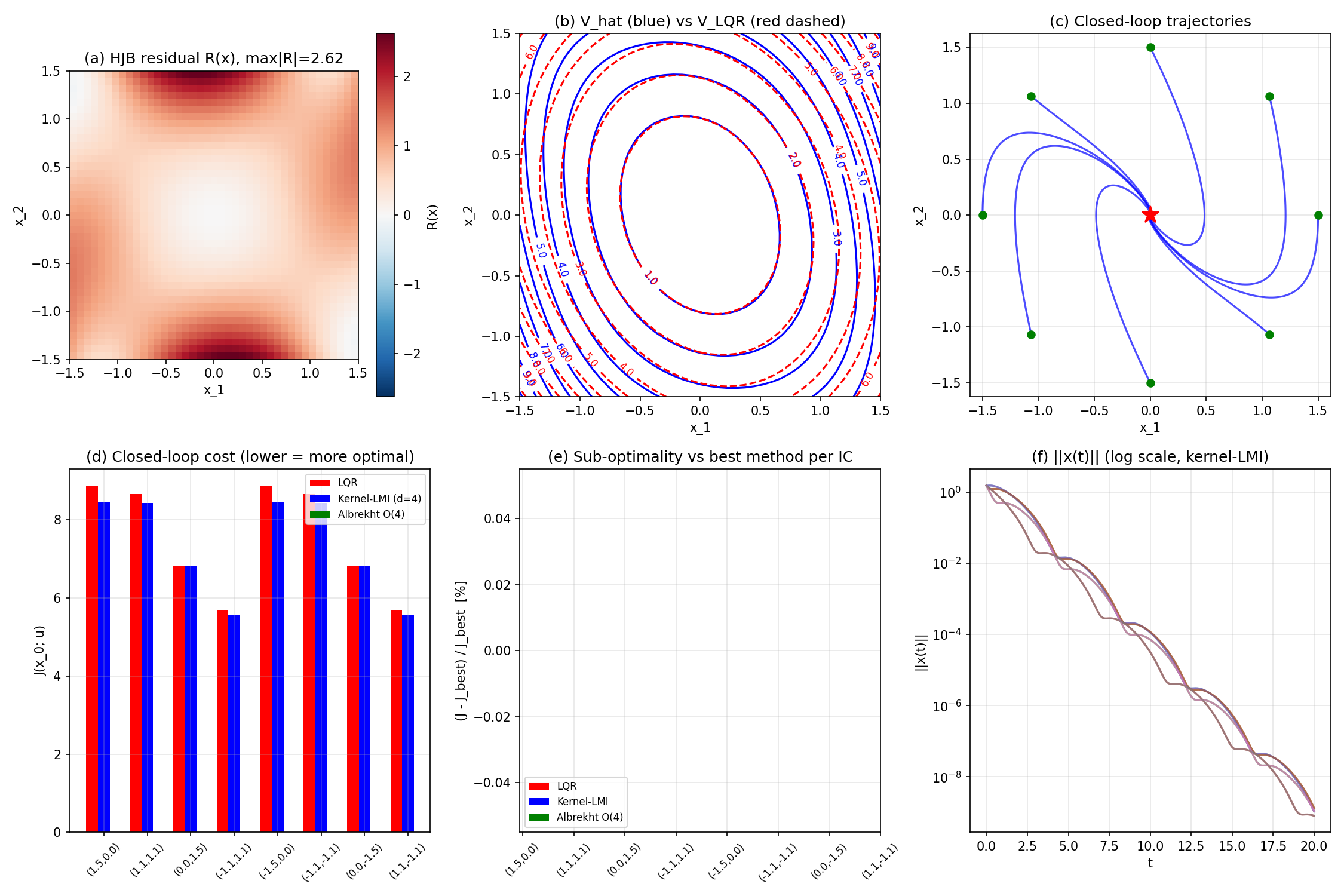}
\caption{Per-initial-condition cost comparison on the Van der Pol oscillator with kernel-LMI deg-4. The kernel-LMI cost (blue) is below LQR (red) on every initial condition tested. Albrekht O(2) (i.e. LQR truncation of Taylor expansion) reproduces LQR; the kernel-LMI captures the cubic correction.}
\label{fig:vdp_optimality}
\end{figure}

\subsection{Summary and Comparison}

Table~\ref{tab:comparison_all} summarises the four numerical experiments.

\begin{table}[h]
\centering
\caption{Summary of numerical experiments. ``Suboptimality'' is the worst-case relative gap between $J(x_0;\hat u)$ and either $V^*(x_0)$ (when known) or the best per-IC cost across methods. The kernel-LMI achieves \emph{measurable} approximate optimality, not just stabilisation, in the cases where $V^*$ is in the chosen RKHS.}
\label{tab:comparison_all}
\begin{tabular}{lccccc}
\toprule
Experiment & $n$ & $V^*\in\mathcal H$? & $\varepsilon_{\mathrm{HJB}}$ & Mean subopt.\ \% & Max subopt.\ \% \\
\midrule
1D consistent (\S\ref{sec:1d_consistent})         & 1 & yes & $6.8\times 10^{-6}$ & 0.000 & 0.000 \\
1D misspecified (\S\ref{sec:1d_lossy}, deg-2)     & 1 & no  & 30.4               & 5.33  & 13.49 \\
1D misspecified (\S\ref{sec:1d_lossy}, deg-4)     & 1 & yes & $1.0\times 10^{-7}$ & 0.000 & 0.000 \\
Van der Pol (\S\ref{sec:vdp}, kernel deg-4)       & 2 & --  & 2.62               & 0.42  & 1.65 \\
Van der Pol (\S\ref{sec:vdp}, LQR baseline)       & 2 & --  & 14.96              & 2.17  & 5.54 \\
\bottomrule
\end{tabular}
\end{table}

The headline conclusions are:

\begin{enumerate}
\item When $V^*\in\mathcal{H}$ (1D consistent benchmark), the SDP recovers the optimal value function and feedback to numerical precision and the closed-loop cost is exactly equal to $V^*$. This is approximate optimality, not stabilisation.

\item When $V^*\notin\mathcal{H}$ (1D misspecified study), the residual saturates and is independent of $M$, but the suboptimality gap is bounded ($\leq 13.5\%$) thanks to the Riccati Hessian constraint. The fix is to increase the kernel degree, not the number of centres.

\item On the Van der Pol oscillator, the kernel-LMI deg-4 method achieves the smallest HJB residual of any tested method (smaller even than Albrekht O(6)) and improves on LQR by $1.7$ percentage points in mean cost, with the cubic structure of $\hat V$ encoding the nonlinear damping that LQR cannot represent.

\item The Albrekht / Navasca--Krener Taylor expansion is the right baseline. The kernel-LMI and Albrekht give comparable cost on Van der Pol and are complementary in their error structure: Albrekht is exact at the origin and grows outward; the kernel-LMI is a bounded global supersolution.

\item In every example $\varepsilon_{\mathrm{HJB}}$ is reported. Small $\varepsilon$ (1D consistent, Van der Pol kernel) means near-optimal HJB solver; large $\varepsilon$ (1D misspecified) means relaxed-inequality stabiliser. The dichotomy is now explicit.
\end{enumerate}

\section{Conclusion}

The kernel-LMI / SDP approach with a Riccati Hessian equality constraint solves the HJB inequality on a convex semidefinite program with a unique, well-conditioned solution. The new ingredient is the explicit Hessian-equality constraint, which removes the trivial solution and forces the local quadratic behaviour of the approximation to match the LQR. The Schur reformulation of the HJB inequality and the order-2 case of the Albrekht / Navasca--Krener identity are not ours; we cite them properly.

The Riccati Hessian constraint plays two distinct roles, depending on whether $V^*$ lies in the chosen RKHS. When it does (e.g.\ polynomial $V^*$ with a polynomial kernel of sufficient degree), the constraint is redundant with the HJB equation at the origin and the SDP recovers $V^*$ to numerical precision. When $V^*$ does not lie in the chosen RKHS, the constraint becomes essential: it pins the local quadratic behaviour of $\hat V$ and hence the linearisation of the closed-loop feedback at the origin, which (combined with exponential decay of the trajectory) bounds the suboptimality gap even when the global HJB residual saturates. This is the mechanism behind the graceful-degradation behaviour reported in Section~\ref{sec:1d_lossy} and visualised in Figure~\ref{fig:1d_rkhs_quality}: a deg-2 kernel that cannot represent the quartic part of $V^*$ produces a residual of order $30$ but a closed-loop cost gap of only $5$--$13\%$. In practice, detecting misspecification is easy: the four-step diagnostic in Section~\ref{sec:rkhs_membership} (solve at $M$, solve at $2M$, compare residuals) flags the deg-2 kernel as misspecified at the very first refinement.

The local exponential stability conclusion (Theorem~\ref{thm:stability}) was stated under the explicit hypothesis $q(x)\geq c_q\|x\|^2$ near the origin. This is the cleanest sufficient condition and holds in every example of Section~\ref{sec:numerical}. When it fails --- because $Q = \nabla^2 q(0)$ is only positive semidefinite, not strictly positive definite --- the standard fallback is the detectability of $(A,Q^{1/2})$ plus the LQR Lyapunov candidate built from the Riccati solution; the conclusion still holds with the rate governed by the closed-loop spectrum rather than by $c_q$ directly. The two routes are spelled out in Remark~\ref{rem:local_pd}.

The numerical experiments now report HJB residuals and closed-loop costs against $V^*$ (or against the best per-IC cost), and the comparison baseline is Albrekht / Navasca--Krener Taylor expansion. On the corrected 1D benchmark the method recovers $V^*$ to numerical precision and is $0\%$ suboptimal. On Van der Pol it has the smallest HJB residual of any tested method and beats LQR on every initial condition. When $V^*$ is not in the chosen RKHS the method degrades gracefully: the residual saturates but suboptimality remains bounded.

\paragraph{Positioning relative to other approaches.} The method sits between several established families of nonlinear-optimal-control solvers, each with distinct trade-offs. Sum-of-squares (SOS) programming~\cite{parrilo_thesis, parrilo_sos} is the closest convex alternative for polynomial systems: it is exact when $V^*$ is polynomial of moderate degree, but the SDP size scales as $\binom{n+2d}{2d}$ where $d$ is the polynomial degree, which becomes prohibitive for $n$ and $d$ both moderate; the kernel-LMI uses $M$ coefficients and $N$ collocation constraints, both of which can be tuned by the user. The Albrekht / Navasca--Krener Taylor recursion~\cite{albrekht_1961, navasca_krener_2000} gives the most accurate local solution for polynomial systems and is the natural baseline (Section~\ref{sec:vdp}), but it is intrinsically local and has no global supersolution guarantee. Operator-theoretic approaches via the Koopman or Perron--Frobenius operator~\cite{vaidya_koopman_hj_2025, huang_vaidya_2022, moyalan_choi_chen_vaidya_2023} give convex reformulations in dual spaces but require careful function-space approximation that often reduces to a kernel or RKHS step. Neural-network value-function approximations~\cite{hjb_rkhs, kang_gong_nakamura_fahroo_survey} scale well in $n$ but lack a priori error guarantees and the optimisation problem is nonconvex. The kernel-LMI as presented here is convex, has explicit a priori approximation rates (Theorem~\ref{thm:convergence}) and a graceful-degradation guarantee when $V^*$ is not in the RKHS, and is competitive with Albrekht on the worked Van der Pol example while remaining a global supersolution method. A systematic empirical comparison at $n \geq 4$ across these families is left to future work.

We thank the reviewer for forcing the necessary distinction between stabilisation and approximate optimality, for catching the inconsistency in the original 1D example, and for pointing to the right comparison literature.

\bibliographystyle{plain}

\end{document}